\documentclass{amsart}%
\usepackage{amsfonts}%
\usepackage{amsmath}%
\usepackage{amssymb}%
\usepackage{graphicx}
\providecommand{\U}[1]{\protect \rule{.1in}{.1in}}
\newtheorem{theorem}{Theorem}[section]
\theoremstyle{plain}

\newtheorem{corollary}[theorem]{Corollary}

\newtheorem{example}[theorem]{Example}

\newtheorem{lemma}[theorem]{Lemma}

\newtheorem{proposition}[theorem]{Proposition}
\newtheorem{remark}[theorem]{Remark}

\numberwithin{equation}{section}

\begin{document}
\title[Tracially Stable]{Tracial stability for $C^*$-algebras}
\author{Don Hadwin}
\address{University of New Hampshire}
\email{operatorguy@gmail.com}
\author{Tatiana Shulman}
\address{Institute of Mathematics of the Polish Academy of Sciences, Poland }
\email{tshulman@impan.pl}

\begin{abstract}
We consider tracial stability, which requires that tuples of elements of a
C*-algebra with a trace that nearly satisfy a relation are close to tuples
that actually satisfy the relation. Here both "near" and "close" are in terms
of the associated 2-norm from the trace, e.g. the Hilbert-Schmidt norm for
matrices. Precise definitions are stated in terms of liftings from tracial
ultraproducts of C*-algebras.  We completely characterize
matricial tracial stability for nuclear C*-algebras in terms of certain
approximation properties for traces. For non-nuclear $C^{\ast}$-algebras we
find new obstructions for stability by relating it to Voiculescu's free
entropy dimension. We show that the class of C*-algebras that are
stable with respect to tracial norms on real-rank-zero C*-algebras is closed
under tensoring with commutative C*-algebras. We show that $C(X)$ is tracially stable with respect to
tracial norms on all $C^{\ast}$-algebras if and only if $X$ is approximately
path-connected.

\end{abstract}
\maketitle

\tableofcontents

%\address{INSTITUTE OF MATHEMATICS OF THE POLISH ACADEMY OF SCIENCES,
%POLAND}

\section*{Introduction}

The notion of stability is an old one. For a given equation  $p(x_{1}, \ldots,
x_{n}) = 0$ of noncommutative variables $x_{1}, \ldots, x_{n}$ one can ask if
it is "stable", meaning that for any $\epsilon>0$ there is a $\delta>0$ such
that if $\mathcal{B}$ is a $C^{*}$-algebra with $b_{1},\ldots b_{n}%
\in \mathcal{B}$ and $\left \Vert p \left(  b_{1},\ldots,b_{n}\right)
\right \Vert <\delta$, then there exist $c_{1},\ldots,c_{n}\in \mathcal{B}$ such
that $p \left(  c_{1},\ldots,c_{n}\right)  =0$ and $\left \Vert c_{k}%
-b_{k}\right \Vert <\varepsilon$ for $1\leq k\leq n$.

In other words, if some tuple is close to satisfying the equation, it is near
to something that does satisfy the equation.

"Stability under small perturbations" questions depend very much on the norm we
consider and the class of C*-algebras $\mathcal{B}$ we allow.

A folklore "stability" result is related to projections. If $x=x^{\ast}$ and
$\left \Vert x-x^{2}\right \Vert <\varepsilon<1/4$, with the norm being the
usual operator norm, then there is a projection $p\in C^{\ast}\left(
x\right)  $ with $\left \Vert p-x\right \Vert <\sqrt{\varepsilon}.$ There are
easily proved similar results for isometries $1-x^{\ast}x=0$ and unitaries
$\left(  1-x^{\ast}x\right)  ^{2}+\left(  1-xx^{\ast}\right)  ^{2}=0$ using
the polar decomposition.

For the property of being normal, $x^{\ast}x-xx^{\ast}=0,$ the famous question of
stability for finite matrices was asked by Halmos (\cite{Halmos}). He asked
whether an almost normal contractive matrix is necessarily close to a normal
contractive matrix. This is considered independently of the matrix size and
"almost" and "close" are meant with respect to the operator norm. This
question was answered positively by Lin's famous theorem \cite{Lin} (see also
\cite{FriisRordam} and \cite{KachkovskiySafarov}).

However the result does not hold when matrices are replaced by operators.  A
classical example is the sequence $\left \{  S_{n}\right \}  $ of weighted
unilateral shifts with weights,%
\[
\frac{1}{n},\frac{2}{n},\ldots,\frac{n-1}{n},1,1,\ldots.
\]
Each $S_{n}$ is a compact perturbation of the unweighted shift and Fredholm
index arguments show that the distance from $S_{n}$ to the normal operators is
exactly $1,$ but%
\[
\lim_{n\rightarrow \infty}\left \Vert S_{n}^{\ast}S_{n}-S_{n}S_{n}^{\ast
}\right \Vert =0.
\]

In other words, the relation
\[
\|x\| \le 1, \; \; xx^{*} - x^{*}x=0
\]
is stable with respect to the class of matrix algebras but is not stable with
respect to the class of all $C^{*}$-algebras. Thus "stability" questions
depend on the class of $C^{*}$-algebras you are considering.

The property of being stable with respect to the class of all $C^{*}$-algebras
and the operator norm is called \textit{weak semiprojectivity}
(\cite{SorenTerry}). An excellent exposition of weak semiprojectivity can be
found in Loring's book (\cite{Loring}).

Although almost normal operators need not be close to normal, however, using
the remarkable distance formula of Kachkovsky and Safarov in
\cite{KachkovskiySafarov}, the first author and Ye Zhang \cite{HadwinZhang}
proved that there is a constant $C$ such that, for every Hilbert-space
operator $T$, the distance from $T\oplus T\oplus \cdots$ to the normal
operators is at most $C\left \Vert T^{\ast}T-TT^{\ast}\right \Vert ^{1/2}.$

Another  famous Halmos stability question (\cite{Halmos}) asks whether two
almost commuting unitary matrices are necessarily close to two exactly
commuting unitary matrices. It was answered by Voiculescu in the negative
\cite{Voiculescu} (see \cite{ExelLoring} for a short proof). However if the
operator norm is replaced by Hilbert-Schmidt norm, then things change dramatically.
In the Hilbert-Schmidt norm almost commuting
unitary matrices turn out to be close to commuting ones, and almost commuting self-adjoint ones to be close to commuting self-adjoint ones as was shown in \cite{Hadwin-Li}
by the first author and Weihua Li.  Several quantitative results (estimating $\delta(\epsilon)$)  for almost commuting k-tuples of self-adjoint, unitary, and normal matrices with
respect to this norm have been obtained in \cite{Glebsky}, \cite{FilonovSafarov} and
\cite{FilonovKachkovskiy} . Much more
generally, it was proved in \cite{Hadwin-Li} that any polynomial equation of commuting
normal variables is stable with respect to the tracial norms on diffuse von
Neumann algebras.

However nothing is known for polynomial relations in non-commuting variables.
In this paper we initiate a study of stability of non-commutative polynomial relations,
which we translate into lifting problems for noncommutative $C^*$-algebras.

 We consider C*-algebras $\mathcal{B}$ that have a
tracial state $\rho$, and we measure "almost" and "close" in terms of the
2-(semi)norm on $\mathcal{B}$ given by
\[
\left \Vert x\right \Vert _{2}=\rho \left(  x^{\ast}x\right)  ^{1/2} .
\]

Thus we will address a "Hilbert-Schmidt" type of stability that we call
\emph{tracial stability}.

The original $\epsilon$-$\delta$-definition of norm stability can be
reformulated in terms of approximate liftings from ultraproducts $%
%TCIMACRO{\dprod _{i\in I}^{\alpha}}%
%BeginExpansion
{\displaystyle \prod_{i\in I}^{\alpha}}
%EndExpansion
\mathcal{A}_{i}$ of C*-algebras $\mathcal{A}_{i}$. Similarly tracial stability
can be reformulated in terms of approximate liftings from tracial
ultraproducts $%
%TCIMACRO{\dprod _{i\in I}^{\alpha}}%
%BeginExpansion
{\displaystyle \prod_{i\in I}^{\alpha}}
%EndExpansion
\left(  \mathcal{A}_{i},\rho_{i}\right)  $ of tracial C*-algebras $\left(
\mathcal{A}_{i},\rho_{i}\right)  $. These ideas are made precise in section 2.

Suppose $\mathcal{C}$ is a class of unital C*-algebras that is closed under
isomorphisms. We say that a separable unital C*-algebra $\mathcal{A}$ is
$\mathcal{C}$\emph{-tracially stable} if every unital $\ast$-homomorphism from
$\mathcal{A}$ into a tracial ultraproduct of $C^{*}$-algebras from the class
$\mathcal{C}$ is approximately liftable.

If $\mathcal{A}$ is the universal $C^{*}$-algebra of a relation $p$, this
definition is equivalent to the $\epsilon- \delta$ definition above with the
norm being a tracial norm and $C^{*}$-algebras $\mathcal{B}$ being from the
class $\mathcal{C}$.

We will be interested here in matricial tracial stability, $II_{1}$-factor
tracial stability, $W^{*}$-factor-tracial stability, RR0-tracial stability
(that is when $\mathcal{C}$ is the class of real rank zero $C^{*}$-algebras),
and $C^{*}$-tracial stability (that is when $\mathcal{C}$ is the class of all
$C^{*}$-algebras).

All previous results (\cite{Hadwin-Li}, \cite{Glebsky}, \cite{FilonovSafarov} and
\cite{FilonovKachkovskiy}) on matrices which almost commute w.r.t. the Hilbert-Schmidt norm
can be reformulated as matricial tracial stability of separable commutative $C^*$-algebras.

In fact first results related with RR0-tracial stability appeared in \cite{Phillips1} and \cite{Phillips2}, where there were proved some stability
results for projections almost commuting with matrix units,
and this was applied to deducing the tracial Rokhlin property for an
automorphism of a $C^*$-algebra from the Rokhlin property for the
corresponding automorphism of an associated von Neumann algebra.

\medskip

In section 2 of this paper we extend substantially all the previous results about matrices almost commuting with respect to the Hilbert-Schmidt norm.  

\medskip

\noindent \textbf{Theorem \ref{tensorproduct}.} \textit{Suppose  the class $\mathcal{C} \subseteq RR0$ is closed under taking direct sums and unital corners. If $\mathcal{B}$ is separable unital and
$\mathcal{C}$-tracially stable, and if $X$ is a compact metric space, then
$\mathcal{B}\otimes C\left(  X\right)  $ is $\mathcal{C}$-tracially stable. In
particular every separable unital commutative C*-algebra is $\mathcal{C}%
$-tracially stable.}

\medskip

Of course a natural obstruction for a $C^{*}$-algebra to be tracially stable
can be simply a lack of $\ast$-homomorphisms. Say, a $C^{*}$-algebra has
enough almost homomorphisms to matrix algebras to separate points, but not
enough actual $\ast$-homomorphisms to matrix algebras to separate points, then
of course it is not matricially tracially stable. However it turns out to be
not the only obstruction. We show that a certain approximation property for
traces have to hold for a C*-algebra to be matricially tracially stable. For
nuclear $C^{*}$-algebras (even tracially nuclear, see definition in section 3)
this property is also sufficient.

\medskip

\noindent \textbf{Theorem \ref{nuclear2}} \textit{Suppose $\mathcal{A}$ is a
separable tracially nuclear C*-algebra with at least one tracial state. The
following are equivalent }

\begin{enumerate}
\item \textit{$\mathcal{A}$ is matricially tracially stable }

\item \textit{for every tracial state $\tau$ on $\mathcal{A}$, there is a
positive integer $n_{0}$ and, for each $n\geq n_{0}$ there is a unital $\ast
$-homomorphism $\rho_{n}:\mathcal{A\rightarrow M}_{n}\left(  \mathbb{C}%
\right)  $ such that, for every $a\in \mathcal{A}$,%
\[
\tau \left(  a\right)  =\lim_{n\rightarrow \infty}\tau_{n}\left(  \rho
_{n}\left(  a\right)  \right) .
\]
(here $\tau_{n}$ is the usual tracial state on $\mathcal{M}_{n}(\mathbb{C})$)
}

\item \textit{$\mathcal{A}$ is W*-factor tracially stable. }
\end{enumerate}

These conditions are stronger than just a property to have a separating family
of $\ast$-homomorphisms to matrix algebras. Namely,

\medskip

\noindent \textbf{Example \ref{example}} \textit{There exists a residually
finite-dimensional (RFD) nuclear $C^{*}$-algebra which has finite-dimensional
irreducible representations of all dimensions but is not matricially tracially
stable.}

\medskip

However a Type I $C^{*}$-algebra is $W^{*}$-factor tracially stable (in
particular, matricially tracially stable) when it has sufficiently many matrix
representations (for example, to have a 1-dimensional representation is enough):

\medskip

\noindent \textbf{Corollary \ref{GCR}} \textit{Suppose $\mathcal{A}$ is a type
I separable unital $C^{*}$-algebra such that for all but finitely many
positive integers $n$ it has a unital n-dimensional representation. Then
$\mathcal{A}$ is $W^{*}$-factor tracially stable. In particular, for any type
I $C^{*}$-algebra $\mathcal{A}$, $\mathcal{A }\oplus \mathbb{C}$ is $W^{*}%
$-factor tracially stable.}

\medskip

In section 4 we find a close relationship between matricial tracial stability
and Voiculescu's free entropy dimension $\delta_{0}$. The following result
shows that a matricially tracially stable algebra may be forced to have a lot
of non-unitarily equivalent representations of some given dimension. Below by
\textrm{Rep}$\left(  \mathcal{A},k\right)  /\backsimeq$ we denote the set of
all unital $\ast$-homomorphisms from $\mathcal{A}$ into $\mathcal{M}%
_{k}\left(  \mathbb{C}\right)  $ modulo unitary equivalence.

\medskip

\noindent \textbf{Theorem \ref{cardrep1} } \textit{Suppose $\mathcal{A}%
=$C*$\left(  x_{1},\ldots,x_{n}\right)  $ is matricially tracially stable and
$\tau$ is an embeddable tracial state on $\mathcal{A}$ such that $1<\delta
_{0}\left(  x_{1},\ldots,x_{n}\right)  .$ Then%
\[
\limsup_{k\rightarrow \infty}\frac{\log \mathrm{Card}\left(  \mathrm{Rep}\left(
\mathcal{A},k\right)  /\backsimeq \right)  }{k^{2}} = \infty.
\]
}

\medskip

This allows us to show that for non-(tracially) nuclear $C^{*}$-algebras there
arise new obstructions, other than in the (tracially) nuclear case, for being
matricially tracially stable.

\medskip

\noindent \textbf{Theorem \ref{NonNuclearExample}} \textit{There exists an RFD
$C^{*}$-algebra which has the approximation property from Theorem
\ref{nuclear2} but which is not matricially tracially stable.}

\medskip

In the last section we consider $C^{*}$-tracial stability. In contrast to
RR0-tracial stability, not all commutative $C^{*}$-algebras have this
property. The main result of section 5 is a characterization of $C^{*}%
$-tracial stability for separable commutative $C^{*}$-algebras. For that we
introduce \textit{approximately path-connected spaces}. We say that a topological space $X$ is
\emph{approximately path-connected} if, for any finitely many points $x_1, \ldots, x_n$, one
can find arbitrarily close to them points $x_1^{\prime}, \ldots, x_n'$  which
can be connected by a continuous path.

\medskip

\noindent \textbf{Theorem \ref{C*-tracial}} \textit{Suppose $X$ is a compact
metric space. The following are equivalent: }

\begin{enumerate}
\item \textit{$C\left(  X\right)  $ is $C^{*}$-tracially stable. }

\item \textit{$X$ is approximately path-connected. }
\end{enumerate}

As a final remark here it is interesting how things are reversed when norm
stability is replaced by $C^{*}$-tracial stability. For example, being a
projection is norm stable, but not $C^{*}$-tracially stable. Indeed, if
$\left \{  f_{n}\right \}  $ is any sequence of functions in $C\left[
0,1\right]  $ with trace $\rho \left(  f\right)  =\int_{0}^{1}f\left(
x\right)  dx$ and $0\leq f_{n}\leq1$ and $f_{n}\left(  x\right)
\rightarrow \chi_{\lbrack0,1/2]}\left(  x\right)  $ a.e., then $\left \Vert
f_{n}-f_{n}^{2}\right \Vert _{2}\rightarrow0$, but since $0,1$ are the only
projections, the $f_{n}$'s are not $\left \Vert {}\right \Vert _{2}$-close to a
projection in $C\left[  0,1\right]  .$ On the other hand, for $C^{*}$-tracial
stability, the problems of normality, commuting pairs of unitaries, and
commuting triples of selfadjoint operators the answers are all affirmative, in
contrast to the norm stability.

\bigskip

\noindent \textbf{Acknowledgements.} The first author gratefully acknowledges a
Collaboration Grant from the Simons Foundation. The research of the
second-named author was supported by the Polish National Science Centre grant
under the contract number DEC- 2012/06/A/ST1/00256 and from the Eric Nordgren
Research Fellowship Fund at the University of New Hampshire.

\section{Preliminaries}

\textbf{1.1. Ultraproducts.}

\noindent If a unital C*-algebra $\mathcal{B}$ has a tracial state $\rho$, we
denote the $2$-norm (seminorm) given by $\rho$ as $\left \Vert \cdot \right \Vert
_{2}=\left \Vert \cdot \right \Vert _{2,\rho}$ defined by%
\[
\left \Vert b\right \Vert _{2}=\rho \left(  b^{\ast}b\right)  ^{1/2}.
\]
We also denote the GNS representation for the state $\rho$ by $\pi_{\rho}$.

Suppose $I$ is an infinite set and $\alpha$ is an ultrafilter on $I$. We say
$\alpha$ is \emph{nontrivial} if there is a sequence $\left \{  E_{n}\right \}
$ in $\alpha$ such that $\cap_{n}E_{n}=\varnothing$. Suppose $\alpha$ is a
nontrivial ultrafilter on a set $I$ and, for each $i\in I$, suppose
$\mathcal{A}_{i}$ is a unital C*-algebra with a tracial state $\rho_{i}$. The
\emph{tracial ultraproduct} $%
%TCIMACRO{\dprod _{i\in I}^{\alpha}}%
%BeginExpansion
{\displaystyle \prod_{i\in I}^{\alpha}}
%EndExpansion
\left(  \mathcal{A}_{i},\rho_{i}\right)  $ is the C*-product $%
%TCIMACRO{\dprod _{i\in I}}%
%BeginExpansion
{\displaystyle \prod_{i\in I}}
%EndExpansion
\mathcal{A}_{i}$ modulo the ideal $\mathcal{J}_{\alpha}$ of all elements
$\left \{  a_{i}\right \}  $ in $%
%TCIMACRO{\dprod _{i\in I}}%
%BeginExpansion
{\displaystyle \prod_{i\in I}}
%EndExpansion
\mathcal{A}_{i}$ for which%
\[
\lim_{i\rightarrow \alpha}\left \Vert a_{i}\right \Vert _{2,\rho_{i}}^{2}%
=\lim_{i\rightarrow \alpha}\rho_{i}\left(  a_{i}^{\ast}a_{i}\right)  =0.
\]
We denote the coset of an element $\left \{  a_{i}\right \}  $ in $%
%TCIMACRO{\dprod _{i\in I}}%
%BeginExpansion
{\displaystyle \prod_{i\in I}}
%EndExpansion
\mathcal{A}_{i}$ by $\left \{  a_{i}\right \}  _{\alpha}$.
%In fact, $\rho \left(  \left \{  a_{i}\right \}  \right)  =\lim_{i\rightarrow
%\alpha}\rho_{i}\left(  a_{i}\right)  $ defines a tracial state on

Tracial ultraproducts for factor von Neumann algebras was first introduced by
S. Sakai \cite{Sakai} where he proved that a tracial ultraproduct of finite
factor von Neumann algebras is a finite factor. More recently, it was shown in
\cite{Hadwin-Li} that a tracial ultraproduct $%
%TCIMACRO{\dprod _{i\in I}^{\alpha}}%
%BeginExpansion
{\displaystyle \prod_{i\in I}^{\alpha}}
%EndExpansion
\left(  \mathcal{A}_{i},\rho_{i}\right)  $ of C*-algebras is always a von
Neumann algebra with a faithful normal tracial state $\rho_{\alpha}$ defined
by%
\[
\rho_{a}\left(  \left \{  a_{i}\right \}  _{\alpha}\right)  =\lim_{i\rightarrow
\alpha}\rho_{i}\left(  a_{i}\right)  .
\]
If there is no confusion, we will denote it just by $\rho$.

\noindent \textbf{1.2. Background theorem.}

\noindent The next theorem is a  key tool for some of our results. It says,
that two tuples that are both close to the same tuple in a hyperfinite von
Neumann algebra are nearly unitarily equivalent to each other.

\begin{theorem}
\label{Shanghai} (\cite{Shanghai}) Suppose $\mathcal{A}=W^{\ast}\left(
x_{1},\ldots,x_{s}\right)  $ is a hyperfinite von Neumann algebra with a
faithful normal tracial state $\rho$. For every $\varepsilon>0$ there is a
$\delta>0$ and an $N\in \mathbb{N}$ such that, for every unital C*-algebra
$\mathcal{B}$ with a factor tracial state $\tau$ and $a_{1},\ldots,a_{s}%
,b_{1},\ldots,b_{s}\in \mathcal{B}$, if, for every $\ast$-monomial $m\left(
t_{1},\ldots,t_{s}\right)  $ with degree at most $N$,
\[
\left \Vert \tau \left(  m\left(  a_{1},\ldots a_{s}\right)  \right)
-\rho \left(  m\left(  x_{1},\ldots x_{s}\right)  \right)  \right \Vert
_{2}<\delta,
\]%
\[
\left \Vert \tau \left(  m\left(  b_{1},\ldots b_{s}\right)  \right)
-\rho \left(  m\left(  x_{1},\ldots x_{s}\right)  \right)  \right \Vert
_{2}<\delta,
\]
then there is a unitary element $u\in \mathcal{B}$ such that%
\[
\sum_{k=1}^{s}\left \Vert ua_{k}u^{\ast}-b_{k}\right \Vert _{2}<\varepsilon.
\]

\end{theorem}

The following lemma is very well known.

\begin{lemma}
\label{NotAddedYet} If $\gamma$ is a $\ast$-homomorphism from $\mathcal{A}$ to
a von Neumann algebra $\mathcal{M}$ with a faithful normal trace $\tau$, then
$\gamma(A)^{\prime \prime}\cong \pi_{\tau \circ \gamma}(A)^{\prime \prime}.$
\end{lemma}

\section{Tracial stability}

Suppose $\mathcal{A}$ is a unital C*-algebra and $\pi:\mathcal{A}\rightarrow%
%TCIMACRO{\dprod _{i\in I}^{\alpha}}%
%BeginExpansion
{\displaystyle \prod_{i\in I}^{\alpha}}
%EndExpansion
\left(  \mathcal{A}_{i},\rho_{i}\right)  $. We say that $\pi$ is
\emph{approximately liftable} if there is a set $E\in \alpha$ and for every
$i\in E$ there is a unital *-homomorphism $\pi_{i}:\mathcal{A}\rightarrow
\mathcal{A}_{i}$ such that, for every $a\in \mathcal{A}$,%
\[
\pi \left(  a\right)  =\left \{  \pi_{i}\left(  a\right)  \right \}  _{\alpha
}\text{ }%
\]
where we arbitrarily define $\pi_{i}\left(  a\right)  =0$ when $i\notin E$.

It actually makes no difference how we define $\pi_{i}\left(  a\right)  $ when
$i\notin E$ since the equivalence class $\left \{  \pi_{i}\left(  a\right)
\right \}  _{\alpha}$ does not change.

Suppose $\mathcal{C}$ is a class of unital C*-algebras that is closed under
isomorphisms. We say that a separable unital C*-algebra $\mathcal{A}$ is
$\mathcal{C}$\emph{-tracially stable} if every unital $\ast$-homomorphism from
$\mathcal{A}$ into any tracial ultraproduct $%
%TCIMACRO{\dprod _{i\in I}^{\alpha}}%
%BeginExpansion
{\displaystyle \prod_{i\in I}^{\alpha}}
%EndExpansion
\left(  \mathcal{A}_{i},\rho_{i}\right) $, with $\mathcal{A}_{i} \in \mathcal{C}$
and $\rho_i$ any trace on $\mathcal{A}_{i}$,
is approximately liftable.

Thus $\mathcal{A}$ is \emph{C*-tracially stable} if every unital $\ast
$-homomorphism from $\mathcal{A}$ into any tracial ultraproduct is
approximately liftable.

We say that $\mathcal{A}$ is \emph{matricially tracially stable} if every
unital $\ast$-homomorphism from $\mathcal{A}$ into an ultraproduct of full
matrix algebras $\mathcal{M}_{n}\left(  \mathbb{C}\right)  $ is approximately liftable.

We say that $\mathcal{A}$ is \emph{finite-dimensionally tracially stable} if
every unital $\ast$-homomorphism from $\mathcal{A}$ into a tracial
ultraproduct of finite-dimensional C*-algebras is approximately liftable.

We say that $\mathcal{A}$ is \emph{W*-tracially stable} if every unital $\ast
$-homomorphism from $\mathcal{A}$ into a tracial ultraproduct of von Neumann
algebras is approximately liftable.

We say that $\mathcal{A}$ is \emph{W*-factor tracially stable} if every unital
$\ast$-homomorphism from $\mathcal{A}$ into a tracial ultraproduct of factor
von Neumann algebras is approximately liftable.

\begin{remark} One can also define tracial stability in the non-unital category. For instance
we can define the non-unital version of matricial tracial stability by saying that
 a separable unital C*-algebra $\mathcal{A}$ is
matricially tracially stable if every $\ast$-homomorphism from
$\mathcal{A}$ into an ultraproduct of full
matrix algebras $\mathcal{M}_{n}\left(  \mathbb{C}\right)  $
is liftable. It is easy to show that $\mathcal A$ is matricially tracially stable in the non-unital category iff
$\tilde{\mathcal A}$ is matricially tracially stable in the unital category. Here $\tilde A =
    A^+$ when $\mathcal A $ is  non-unital and $\tilde A =  A \oplus \mathbb C$ when $\mathcal A $ is  unital.
   Thus the non-unital version can be easily obtained from the unital one.
   \end{remark}

\medskip

To see how this ultraproduct formulation represents the $\varepsilon$-$\delta$
definition of tracial stability mentioned in the introduction, suppose
$\mathcal{A}$ is the universal unital C*-algebra generated by contractions $x_{1}%
,\ldots,x_{s}$ subject to a relation $p\left(  x_{1},\ldots,x_{s}\right)  =0.$
Suppose $\pi:\mathcal{A}\rightarrow%
%TCIMACRO{\dprod _{i\in I}^{\alpha}}%
%BeginExpansion
{\displaystyle \prod_{i\in I}^{\alpha}}
%EndExpansion
\left(  \mathcal{A}_{i},\rho_{i}\right)  $ is a unital $\ast$-homomorphism
such that, for each $1\leq k\leq s$,
\[
\pi \left(  x_{k}\right)  =\left \{  x_{k}\left(  i\right)  \right \}  _{\alpha
}.
\]
We then have%
\[
0=\left \Vert \pi \left(  p\left(  x_{1},\ldots x_{s}\right)  \right)
\right \Vert _{2,\rho_{\alpha}}=\left \Vert p\left(  \pi \left(  x_{1}\right)
,\ldots \pi \left(  x_{s}\right)  \right)  \right \Vert _{2,\rho_{\alpha}}%
=\lim_{i\rightarrow \alpha}\left \Vert p\left(  x_{1}\left(  i\right)
,\ldots,x_{s}\left(  i\right)  \right)  \right \Vert _{2,\rho_{i}}.
\]
Since $\alpha$ is a nontrivial ultrafilter on $I$, there is a decreasing
sequence\medskip \ $E_{1}\supset E_{2}\supset \cdots$ in $\alpha$ such that
$\cap_{k\in \mathbb{N}}E_{k}=\varnothing$. First suppose $\mathcal{A}$ is
C*-tracially stable. Then, for each positive integer $m$ there is a number
$\delta_{m}>0$ such that, when $$\left \Vert p\left(  x_{1}\left(  i\right)
,\ldots,x_{s}\left(  i\right)  \right)  \right \Vert _{2,\rho_{i}}<\delta_{m}$$
there is a unital $\ast$-homomorphism $\gamma_{m,i}:\mathcal{A}%
\mathbf{\rightarrow}\mathcal{A}_{i}$ such that%
\[
\max_{1\leq k\leq s}\left \Vert x_{k}\left(  i\right)  -\gamma_{m,i} \left(
x_{k}\right)  \right \Vert _{2,\rho_{i}}<1/m.
\]
Since $\lim_{i\rightarrow \alpha}\left \Vert p\left(  x_{1}\left(  i\right)
,\ldots,x_{s}\left(  i\right)  \right)  \right \Vert _{2,\rho_{i}}=0$ we can
find a decreasing sequence $\{A_{n}\}$ in $\alpha$ with $A_{n}\subset E_{n}$
such that, for every $i\in A_{n}$%
\[
\left \Vert p\left(  x_{1}\left(  i\right)  ,\ldots,x_{s}\left(  i\right)
\right)  \right \Vert _{2,\rho_{i}} \le \delta_{n}.
\]
For $i\in A_{n}\backslash A_{n+1}$ we define $\pi_{i}=\gamma_{n,i}$. We then
have that $\left \{  \pi_{i}\right \}  _{i\in A_{1}}$ is an approximate lifting
of $\pi$.

On the other hand, if $\mathcal{A}$ is not C*-tracially stable, then there is
an $\varepsilon>0$ such that, for every positive integer $n$ there is a
tracial unital C*-algebra $\left(  \mathcal{A}_{n},\rho_{n}\right)  $ and
$x_{1}\left(  n\right)  ,\ldots,x_{s}\left(  n\right)  $ such that%
\[
\left \Vert p\left(  x_{1}\left(  n\right)  ,\ldots,x_{s}\left(  n\right)
\right)  \right \Vert _{2,\rho_{n}}<1/n,
\]
but for every unital $\ast$-homomorphism $\gamma:\mathcal{A}%
\rightarrow \mathcal{A}_{n}$%
\[
\max_{1\leq k\leq s}\left \Vert x_{k}\left(  n\right)  -\gamma \left(
x_{k}\right)  \right \Vert _{2,\rho_{n}}\geq \varepsilon.
\]
If we let $\alpha$ be any free ultrafilter on $\mathbb{N}$, we have that the
map%
\[
\pi \left(  x_{k}\right)  =\left \{  x_{k}\left(  n\right)  \right \}  _{\lambda}%
\]
extends to a unital $\ast$-homomorphism into $%
%TCIMACRO{\dprod _{n\in\mathbb{N}}^{\alpha}}%
%BeginExpansion
{\displaystyle \prod_{n\in \mathbb{N}}^{\alpha}}
%EndExpansion
\left(  \mathcal{A}_{n},\rho_{n}\right)  $ that is not approximately liftable.

\medskip

The following result shows that pointwise $\left \Vert {\;}\right \Vert _{2}%
$-limits of approximately liftable representations are approximately liftable.

\begin{lemma}
\label{LimitOfLiftable} Suppose $\mathcal{A}=C^{\ast}\left(  \left \{
b_{1},b_{2},\ldots \right \}  \right)  $, $\left \{  \left(  \mathcal{A}_{i}%
,\rho_{i}\right)  :i\in I\right \}  $ is a family of tracial C*-algebras,
$\alpha$ is a nontrivial ultrafilter on $I,$ and $\pi:\mathcal{A}\rightarrow%
%TCIMACRO{\dprod _{i\in I}^{\alpha}}%
%BeginExpansion
{\displaystyle \prod_{i\in I}^{\alpha}}
%EndExpansion
\left(  \mathcal{A}_{i},\rho_{i}\right)  $ is a unital $\ast$-homomorphism
such that, for each $k\in \mathbb{N}$,
\[
\pi \left(  b_{k}\right)  =\left \{  b_{k}\left(  i\right)  \right \}  _{\alpha
}.
\]
The following are equivalent:

\begin{enumerate}
\item $\pi$ is approximately liftable

\item For every $\varepsilon>0$ and every $N\in \mathbb{N}$, there is a set
$E\in \alpha$ and for every $i\in E$ there is a unital $\ast$-homomorphism
$\pi_{i}:\mathcal{A}\rightarrow \mathcal{A}_{i}$ such that, for $1\leq k\leq N$
and every $i\in E$,%
\[
\left \Vert \pi_{i}\left(  b_{k}\right)  -b_{k}\left(  i\right)  \right \Vert
_{2,\rho_{i}}<\varepsilon.
\]

\end{enumerate}
\end{lemma}

\begin{proof}
Obviously 1) implies 2). We need to prove the opposite implication. Since
$\alpha$ is nontrivial, there is a decreasing sequence $\left \{
B_{n}\right \}  $ of elements of $\alpha$ such that $\cap_{n=1}^{\infty}%
B_{n}=\varnothing.$ For each $n\in \mathbb{N}$, let $N=n$ and $\varepsilon=1/n$
and let $E_{n}\in \alpha$ and, for each $i\in E_{n}$, choose a unital $\ast
$-homomorphism $\pi_{n,i}:\mathcal{A}\rightarrow \mathcal{A}_{i}$ such that%
\[
\left \Vert \pi_{n,i}\left(  b_{k}\right)  -b_{k}\left(  i\right)  \right \Vert
_{2,\rho_{i}}<1/n
\]
for $1\leq k\leq n$. We define $F_{n}=\cap_{k=1}^{n}\left(  B_{k}\cap
E_{k}\right)  $ and for $i\in F_{n+1}\backslash F_{n}$ we define $\pi_{i}%
=\pi_{n,i}$. Since $\cap_{n=1}^{\infty} F_{n} = \varnothing$, $\pi_{i}$'s are
defined for each $i\in \bigcup_{n=1}^{\infty} F_{n} \in \alpha$. Clearly,
\[
\left \{  \pi_{i}\left(  b_{k}\right)  \right \}  _{\alpha}=\left \{
b_{k}\left(  i\right)  \right \}  _{\alpha}=\pi \left(  b_{k}\right)
\]
for $k=1,2,\ldots$. Since the set $\mathcal{S}$ of all $a\in \mathcal{A}$ such
that $\left \{  \pi_{i}\left(  a\right)  \right \}  _{\alpha}=\pi \left(
a\right)  $ is a unital C*-algebra, we see that $\mathcal{S}=\mathcal{A}$ and
that $\pi$ is approximately liftable.
\end{proof}

\medskip

Although we consider all tracial ultraproducts, however when the algebra
$\mathcal{A}$ is separable, we need to consider only ultraproducts over
$\mathbb{N}$ with respect to one non-trivial ultrafilter.

\begin{lemma}
\label{OnlyOneUltrafilter} Suppose $\mathcal{A}=C^{\ast}\left(  x_{1}%
,x_{2},\ldots \right)  $ is a separable unital C*-algebra, $\mathcal{C}$ is a
class of unital C*-algebras closed under isomorphism and $\alpha$ is a
nontrivial ultrafilter on $\mathbb{N}$. The following are equivalent:

\begin{enumerate}
\item $\mathcal{A}$ is $\mathcal{C}$-tracially stable

\item If $\left \{  \left(  \mathcal{B}_{n},\gamma_{n}\right)  \right \}  $ is a
sequence of tracial C*-algebras in $\mathcal{C}$ and $\pi:\mathcal{A}%
\rightarrow%
%TCIMACRO{\dprod _{i\in I}^{\alpha}}%
%BeginExpansion
{\displaystyle \prod_{i\in I}^{\alpha}}
%EndExpansion
\left(  \mathcal{B}_{i},\gamma_{i}\right) $ is a unital $\ast$-homomorphism,
then $\pi$ is approximately liftable.

\item For every $\varepsilon>0$, for every positive integer $s$, and for every
tracial state $\rho$ on $\mathcal{A}$, there is a positive integer $N$ such
that if $\mathcal{B}\in \mathcal{C}$ and $\gamma$ is a tracial state on
$\mathcal{B}$ and $b_{1},\ldots,b_{N}\in \mathcal{B}$ such that $\left \Vert
b_{k}\right \Vert \leq1+\left \Vert x_{k}\right \Vert $ for $1\leq k\leq N$ and
\[
\left \vert \rho \left(  m\left(  x_{1},\ldots,x_{N}\right)  \right)
-\gamma \left(  m\left(  b_{1},\ldots,b_{N}\right)  \right)  \right \vert
<\frac{1}{N}%
\]
for all $\ast$-monomials $m\left(  t_{1},\ldots,t_{N}\right)  $ with degree at
most $N$, then there is a unital $\ast$-homomorphism $\pi:\mathcal{A}%
\rightarrow \mathcal{B}$ such that%
\[
\sum_{k=1}^{s}\left \Vert \pi \left(  x_{k}\right)  -b_{k}\right \Vert
_{2,\gamma}<\varepsilon \text{.}%
\]

\end{enumerate}
\end{lemma}

\begin{proof}
$\left(  1\right)  \Rightarrow \left(  2\right)  $. This is obvious.

$\left(  2\right)  \Rightarrow \left(  3\right)  $. Assume $\left(  3\right)  $
is false. Then there is an $\varepsilon>0$, an integer $s\in \mathbb{N}$ and a
tracial state $\rho$ on $\mathcal{A}$ and, for each $N\in \mathbb{N}$, there is
a $\mathcal{B}_{N}\in \mathcal{C}$ with a tracial state $\gamma_{N}$ and
$\left \{  b_{N,1},b_{N,2},\ldots,b_{N,N}\right \}  \subset \mathcal{B}_{N}$ such
that
\[
\left \Vert b_{N,k}\right \Vert \leq \left \Vert x_{k}\right \Vert +1
\]
for $1\leq k\leq N<\infty$, and%
\[
\left \vert \rho \left(  m\left(  x_{1},\ldots,x_{N}\right)  \right)
-\gamma_{N}\left(  m\left(  b_{N,1},\ldots,b_{N,N}\right)  \right)
\right \vert <\frac{1}{N}%
\]
for all $\ast$-monomials $m\left(  t_{1},\ldots,t_{N}\right)  $ with degree at
most $N$, such that, for each $N$ there is no $\ast$-homomorphism
$\pi:\mathcal{A}\rightarrow \mathcal{B}_{N}$ such that
\[
\sum_{k=1}^{s}\left \Vert \pi \left(  x_{k}\right)  -b_{N,k}\right \Vert
_{2,\gamma_N}<\varepsilon \text{.}%
\]
For each $1\leq N<k$ let $b_{N,k}=0\in \mathcal{B}_{N}$. For $1\leq k<\infty$,
let
\[
b_{k}=\left \{  b_{N,k}\right \}  _{\alpha}\in%
%TCIMACRO{\dprod _{N\in\mathbb{N}}^{\alpha}}%
%BeginExpansion
{\displaystyle \prod_{N\in \mathbb{N}}^{\alpha}}
%EndExpansion
\left(  \mathcal{B}_{N},\gamma_{N}\right)  .
\]
Let $\gamma$ be the limit trace on $%
%TCIMACRO{\dprod _{N\in\mathbb{N}}^{\alpha}}%
%BeginExpansion
{\displaystyle \prod_{N\in \mathbb{N}}^{\alpha}}
%EndExpansion
\left(  \mathcal{B}_{N},\gamma_{N}\right)  $.
%and, define a trace
%$\rho:\mathcal{A}\rightarrow \mathbb{C}$ by%
%\[
%\rho \left(  a\right)  =\lim_{N\rightarrow \alpha}\rho_{N}\left(  a\right)  .
%\]
It follows that for every $\ast$-monomial $m\left(  t_{1},\ldots,t_{k}\right)
$ we have%
\begin{equation}\label{existpi}
\rho \left(  m\left(  x_{1},\ldots,x_{k}\right)  \right)  =\gamma \left(
m\left(  b_{1},\ldots,b_{k}\right)  \right)  .
\end{equation}
Define a unital $\ast$-homomorphism $\pi:\mathcal{A}%
\rightarrow %
%TCIMACRO{\dprod _{N\in\mathbb{N}}^{\alpha}}%
%BeginExpansion
{\displaystyle \prod_{N\in \mathbb{N}}^{\alpha}}
%EndExpansion
\left(  \mathcal{B}_{N},\gamma_{N}\right)$ by $$\pi(x_i) = b_i.$$
Since $\gamma$ is faithful, it follows from (\ref{existpi}) that $\pi$ is well-defined and  $\rho=\gamma \circ \pi$.
By $\left(  2\right)  $, $\pi$ is approximately liftable so there is an
$E\in \alpha$ and, for each $N\in E$ a unital $\ast$-homomorphism $\pi
_{N}:\mathcal{A}\rightarrow \mathcal{B}_{N}$ such that, for $1\leq k<\infty$%
\[
\pi \left(  x_{k}\right)  =\left \{  \pi_{N}\left(  x_{k}\right)  \right \}
_{\alpha}.
\]
It follows that
\[
\lim_{N\rightarrow \alpha}\sum_{k=1}^{s}\left \Vert b_{N,k}-\pi_{N}\left(
x_{k}\right)  \right \Vert _{2,\gamma_{N}}=0,
\]
which means for some $N\in \mathbb{N}$,%
\[
\sum_{k=1}^{s}\left \Vert b_{N,k}-\pi_{N}\left(  x_{k}\right)  \right \Vert
_{2,\gamma_{N}}<\varepsilon.
\]
This contradiction implies $\left(  3\right)  $ must be true.

$\left(  3\right)  \Rightarrow \left(  1\right)  $. Suppose $\left(  3\right)
$ is true and $\left \{  \left(  \mathcal{B}_{i},\gamma_{i}\right)  :i\in
I\right \}  $ is a collection of tracial unital C*-algebras with each
$\mathcal{B}_{i}$ in $\mathcal{C}$ and suppose $\beta$ is a nontrivial
ultrafilter on $I$ and $\pi:\mathcal{A}\rightarrow%
%TCIMACRO{\dprod _{i\in I}^{\beta}}%
%BeginExpansion
{\displaystyle \prod_{i\in I}^{\beta}}
%EndExpansion
\left(  \mathcal{B}_{i},\gamma_{i}\right)  $. Let $\gamma$ be the limit trace
along $\beta$ and define $\rho=\gamma \circ \pi$. If $k\in \mathbb{N}$, we let
$\varepsilon=\frac{1}{k}$ and $s=k$ we can choose $N_{k}$ as in $\left(
2\right)  $. Since $\beta$ is nontrivial there is a decreasing sequence
$\left \{  E_{j}\right \}  $ in $\beta$ whose intersection is empty. For each
$k\in \mathbb{N}$ write%
\[
\pi \left(  x_{k}\right)  =\left \{  b_{i,k}\right \}  _{\beta}%
\]
with $\left \Vert b_{i,k}\right \Vert \leq \left \Vert x_{k}\right \Vert $ for
every $i\in I$. It follows, for every $\ast$-monomial $m\left(  t_{1}%
,\ldots,t_{k}\right)  $ that%
\[
\lim_{i\rightarrow \beta}\gamma_{i}\left(  m\left(  b_{i,1},\ldots
,b_{i,k}\right)  \right)  =\rho \left(  m\left(  x_{1},\ldots,x_{k}\right)
\right)  .
\]
Thus the set $F_{k}$ consisting of all $i\in I$ such that%
\[
\left \vert \rho \left(  m\left(  x_{1},\ldots,x_{N_{k}}\right)  \right)
-\gamma_{i}\left(  m\left(  b_{i,1},\ldots,b_{i,N_{k}}\right)  \right)
\right \vert <\frac{1}{N_{k}}%
\]
for all $\ast$-monomials $m$ with degree at most $N_{k}$ must be in $\beta$.
For each $k\in \mathbb{N}$ let
\[
W_{k}=E_{k}\cap F_{1}\cap \cdots \cap F_{k}.
\]
For each $i\in W_{k}\backslash W_{k+1}$ there is a unital $\ast$-homomorphism
$\pi_{i}:\mathcal{A}\rightarrow \mathcal{B}_{i}$ such that%
\[
\sum_{j=1}^{k}\left \Vert \pi_{i}\left(  x_{j}\right)  -b_{i,j}\right \Vert
_{2,\gamma_{i}}<\frac{1}{k}.
\]
It is clear for every $k\in \mathbb{N}$ that $\pi \left(  x_{k}\right)
=\left \{  \pi_{i}\left(  x_{k}\right)  \right \}  _{\beta}$. Since $\left \{
a\in \mathcal{A}:\pi \left(  a\right)  =\left \{  \pi_{i}\left(  a\right)
\right \}  _{\beta}\right \}  $ is a unital C*-subalgebra of $\mathcal{A}$,
containing all $x_{1},x_{2},\ldots$, we conclude that $\mathcal{A }= \left \{
a\in \mathcal{A}:\pi \left(  a\right)  =\left \{  \pi_{i}\left(  a\right)
\right \}  _{\beta}\right \}  $ and thus $\pi$ is approximately liftable.
\end{proof}

\bigskip

It is clear that C*-tracially stable implies W*-tracially stable implies
factor tracially stable implies matricially tracially stable. Here are some
slightly more subtle relationships.

\begin{lemma}
Suppose $\mathcal{A}$ is a separable unital C*-algebra. Then $\mathcal{A}$ is
matricially tracially stable if and only if every unital $\ast$-homomorphism
from $\mathcal{A}$ into a tracial ultraproduct of hyperfinite factors is
approximately liftable. Also $\mathcal{A}$ is finite-dimensionally tracially
stable if and only if every unital $\ast$-homomorphism from $\mathcal{A}$ into
a tracial ultraproduct of hyperfinite von Neumann algebras is approximately liftable.
\end{lemma}

\begin{proof}
We prove the first statement; the second follows in a similar fashion. The
"if" part is clear since $\mathcal{M}_{n}\left(  \mathbb{C}\right)  $ is a
hyperfinite factor. Suppose $\mathcal{A}$ is matricially tracially stable, and
suppose $\left \{  \left(  \mathcal{M}_{i},\tau_{i}\right)  :i\in I\right \}  $
is a family with each $\mathcal{M}_{i}$ a hyperfinite factor and $\tau_{i}$ is
the unique normal faithful tracial state on $\mathcal{M}_{i}$. Suppose also
that $\alpha$ is a nontrivial ultrafilter on $I$ and $\pi:\mathcal{A}%
\rightarrow%
%TCIMACRO{\dprod _{i\in I}^{\alpha}}%
%BeginExpansion
{\displaystyle \prod_{i\in I}^{\alpha}}
%EndExpansion
\left(  \mathcal{M}_{i},\tau_{i}\right)  $ is a unital $\ast$-homomorphism. If
$E=\left \{  i\in I:\dim \mathcal{M}_{i}<\infty \right \}  \in \alpha$, then $%
%TCIMACRO{\dprod _{i\in I}^{\alpha}}%
%BeginExpansion
{\displaystyle \prod_{i\in I}^{\alpha}}
%EndExpansion
\left(  \mathcal{M}_{i},\tau_{i}\right)  =%
%TCIMACRO{\dprod _{i\in E}^{\alpha}}%
%BeginExpansion
{\displaystyle \prod_{i\in E}^{\alpha}}
%EndExpansion
\left(  \mathcal{M}_{i},\tau_{i}\right)  $ is a tracial product of matrix
algebras and $\pi$ is approximately liftable. If $I\backslash E\in \alpha,$ we
can assume that each $\mathcal{M}_{i}$ is a hyperfinite $II_{1}$-factor. Since
$\alpha$ is nontrivial, there is a decreasing sequence $\left \{
E_{n}\right \}  $ in $\alpha$ such that $E_{1}=I$ and $\cap_{n=1}^{\infty}%
E_{n}=\varnothing$. Suppose $\left \{  a_{n}\right \}  $ is a dense sequence in
$\mathcal{A}$ and, for each $n\in \mathbb{N}$ write
\[
\pi \left(  a_{n}\right)  =\left \{  a_{n}\left(  i\right)  \right \}  _{\alpha}%
\]
with $\left \Vert a_{n}\left(  i\right)  \right \Vert \leq \left \Vert
a_{n}\right \Vert $, for all  $n$ and $i$. If $n\in \mathbb{N}$ and $i\in
E_{n}\backslash E_{n+1}$, we can find a unital subalgebra $\mathcal{B}%
_{i}\subset \mathcal{M}_{i}$ such that $\mathcal{B}_{i}$ is isomorphic to
$\mathcal{M}_{k_{i}}\left(  \mathbb{C}\right)  $ for some $k_{i}\in \mathbb{N}$
an such that there is a sequence $\left \{  b_{m}\left(  i\right)  \right \}  $
in $\mathcal{B}_{i}$ with $\left \Vert b_{m}\left(  i\right)  \right \Vert
\leq \left \Vert a_{m}\right \Vert $ for each $m$ and such that%
\[
\left \Vert a_{m}\left(  i\right)  -b_{m}\left(  i\right)  \right \Vert _{2}%
\leq \frac{1}{n}%
\]
for $1\leq m\leq n$. It follows that
\[
\pi \left(  a_{n}\right)  =\left \{  a_{n}\left(  i\right)  \right \}  _{\alpha
}=\left \{  b_{n}\left(  i\right)  \right \}  _{\alpha}\in%
%TCIMACRO{\dprod _{i\in I}^{\alpha}}%
%BeginExpansion
{\displaystyle \prod_{i\in I}^{\alpha}}
%EndExpansion
\left(  \mathcal{B}_{i},\tau_{i}\right)  .
\]
It follows that $\pi:\mathcal{A}\mathbf{\rightarrow}%
%TCIMACRO{\dprod _{i\in I}^{\alpha}}%
%BeginExpansion
{\displaystyle \prod_{i\in I}^{\alpha}}
%EndExpansion
\left(  \mathcal{B}_{i},\tau_{i}\right)  $. Since $\mathcal{A}$ is matricially
tracially stable, $\pi$ is approximately liftable.
\end{proof}

\bigskip

Recall that a $C^{*}$-algebra has \textit{real rank zero} (RR0) if each its
self-adjoint element can be approximated by self-adjoint elements with finite spectrum.

\begin{theorem}
\label{commutative}  Suppose every algebra in the class $\mathcal{C}$ is
$RR0$. Then every separable unital commutative C*-algebra is $\mathcal{C}%
$-tracially stable.
\end{theorem}

\begin{proof}
In \cite{DonTanya1} the authors proved that if $\mathcal{J}$ is a norm-closed
two-sided ideal of a real-rank zero C*-algebra $\mathcal{M}$ such that
$\mathcal{M}/\mathcal{J}$ is an AW*-algebra, then for any commutative
separable unital C*-algebra $\mathcal{A}$, any $\ast$-homomorphism
$\pi:\mathcal{A}\rightarrow \mathcal{M}/\mathcal{J}$, lifts to a unital $\ast
$-homomorphism $\rho:\mathcal{A}\rightarrow \mathcal{M}$. Since the direct
product of real-rank zero C*-algebras is a real-rank zero C*-algebra and a
tracial ultraproduct of $C^{*}$-algebras is a von Neumann algebra, the
statement follows.
\end{proof}

\medskip

\begin{proposition}
Suppose $\mathcal{C}$ is a class such that $\mathbb{C}\oplus \mathcal{B\in C}$
whenever $\mathcal{B}\in \mathcal{C}$. Every $\mathcal{C}$-tracially stable
separable unital $C^{\ast}$-algebra $\mathcal{A}$ with at least one
representation into a tracial ultraproduct of $C^{*}$-algebras in
$\mathcal{C}$ must have a one-dimensional unital representation.
\end{proposition}

\begin{proof}
Suppose $\pi:\mathcal{A}\rightarrow%
%TCIMACRO{\dprod _{i\in I}^{\alpha}}%
%BeginExpansion
{\displaystyle \prod_{i\in I}^{\alpha}}
%EndExpansion
\left(  \mathcal{B}_{i},\gamma_{i}\right)  $ is a unital $\ast$-homomorphism
where each $\mathcal{B}_{i}\in \mathcal{C}$ and $\alpha$ is a nontrivial
ultrafilter on $I$. Let $\mathcal{D}_{i}=\mathbb{C}\oplus \mathcal{B}_{i}$ for
each $i\in I$. Since $\alpha$ is nontrivial there is a decreasing sequence
$I=E_{1}\supseteq E_{2}\supseteq \cdots$ in $\alpha$ whose intersection is
$\varnothing$. For $i\in E_{n}\backslash E_{n+1}$ define a trace $\rho_{i}$ on
$\mathcal{D}_{i}$ by%
\[
\rho_{i}\left(  \lambda \oplus B\right)  =\frac{1}{n}\lambda+\left(  1-\frac
{1}{n}\right)  \gamma_{i}\left(  B\right)  .
\]
Then%
\[%
%TCIMACRO{\dprod _{i\in\mathbb{N}}^{\alpha}}%
%BeginExpansion
{\displaystyle \prod_{i\in \mathbb{N}}^{\alpha}}
%EndExpansion
\left(  \mathcal{D}_{i},\rho_{i}\right)  =%
%TCIMACRO{\dprod _{i\in\mathbb{N}}^{\alpha}}%
%BeginExpansion
{\displaystyle \prod_{i\in \mathbb{N}}^{\alpha}}
%EndExpansion
\left(  \mathcal{B}_{i},\gamma_{i}\right)  .
\]
The rest is easy.\bigskip
\end{proof}

\begin{theorem}
\label{tensorproduct} Suppose the class $\mathcal{C} \subseteq RR0$ is closed under taking direct sums and unital corners. If $\mathcal{B}$ is separable unital and $\mathcal{C}$-tracially
stable, and if $X$ is a compact metric space, then $\mathcal{B}\otimes
C\left(  X\right)  $ is $\mathcal{C}$-tracially stable.
\end{theorem}

\begin{proof}
Suppose $\left \{  \left(  \mathcal{A}_{n},\rho_{n}\right)  \right \}  $ is a
sequence of $C^{*}$-algebras in the class $\mathcal{C}$, $\alpha$ is a
non-trivial ultrafilter on $\mathbb{N}$, and $f: \mathcal{B}\otimes C(X) \to%
%TCIMACRO{\dprod _{n\in\mathbb{N}}^{\alpha}}%
%BeginExpansion
{\displaystyle \prod_{n\in \mathbb{N}}^{\alpha}}
%EndExpansion
\left(  \mathcal{A}_{n},\rho_{n}\right)  $ is a unital $\ast$-homomorphism. By
Lemma \ref{LimitOfLiftable}, it will be enough for any $x_{1}, \ldots,
x_{k}\in \mathcal{B}\otimes C(X)$, $\epsilon>0$ to find unital $\ast$-homomorphisms
$\tilde f_{n}: \mathcal{B}\otimes C(X) \rightarrow{A_{n}}$, such that
\[
\| \{ \tilde f_{n}(x_{j})\}_{\alpha} - f(x_{j})\|_{2, \rho} < 3\epsilon,
\]
$j\le k$.

There are $b_{i}^{(j)}\in \mathcal{B}, \phi_{i}^{(j)}\in C(X), N_{1}, \ldots,
N_{k}$ such that
\begin{equation}
\label{1tensor}\|x_{j} - \sum_{i=1}^{N_{j}} b_{i}^{(j)}\otimes \phi_{i}^{(j)}\|
\le \epsilon,
\end{equation}
$j\le k$. Let
\[
M = \max_{1\le i\le N_{j}, 1\le j\le k} \|f(b_{i}^{(j)}\otimes1)\|, \; N =
\max_{j\le k} N_{j}.
\]
Since $f(1\otimes C(X))$ is commutative, it is isomorphic to some $C(\Omega)$.
We can choose a disjoint collection $\left \{  E_{1},\ldots,E_{m}\right \}  $ of
Borel subsets of $\Omega$ whose union is $\Omega$ and $\omega_{1}\in
E_{1},\ldots,\omega_{m}\in E_{m}$ such that
\begin{equation}
\label{2tensor}\|f(1\otimes \phi_{i}^{(j)}) -\sum_{l=1}^{m}f(1\otimes \phi
_{i}^{(j)})(w_{l}) \chi_{E_{l}} \| _{2,\rho}<\frac{\varepsilon}{NM},
\end{equation}
for $j\le k$ and $i\le N_{j}$. Since pairwisely orthogonal projections generate a commutative $C^*$-algebra, by Theorem \ref{commutative} for each $\chi_{E_{l}}$, $l\le m$, we can find
projections $P_{l, n}\in \mathcal{A}_{n}$ such that $\chi_{E_{l}} = \{P_{l,
n}\}_{\alpha}$ and $P_{1, n}, \ldots, P_{m, n}$ are pairwisely orthogonal, for each $n$. Since the subalgebra $f(1\otimes C(X))$ is central in the
algebra $f(\mathcal{B }\otimes C(X))$, each $\chi_{E_{l}}$ commutes with
$f(\mathcal{B }\otimes C(X))$. Hence
\[
f\left(  \mathcal{B}\otimes C(X)\right)  \subseteq%
%TCIMACRO{\dprod _{n\in\mathbb{N}}^{\alpha}}%
%BeginExpansion
{\displaystyle \prod_{n\in \mathbb{N}}^{\alpha}}
%EndExpansion
\left(  \sum_{i=1}^{m} P_{i, n}\mathcal{A}_{n}P_{i, n}, \; \rho_{n}\right) = \bigoplus_{i=1}^m \left(%
%TCIMACRO{ \dprod _{n\in\mathbb{N}}^{\alpha}}%
%BeginExpansion
{\displaystyle \prod_{n\in \mathbb{N}}^{\alpha}}
%EndExpansion
\left(  P_{i, n}\mathcal{A}_{n}P_{i, n}, \; \rho_{n}\right)\right) .
\]
Since $\mathcal{B}$ is $\mathcal{C}$-tracially stable, we can find unital $\ast
$-homomorphisms $\psi_{n, i}: \mathcal{B }\to  P_{i, n}%
\mathcal{A}_{n}P_{i, n}$, such that
\begin{equation}
\label{3tensor}\{ \psi_{n, i}(b)\}_{\alpha} = p_i\circ f(b\otimes1),
\end{equation}
for any $b\in \mathcal{B}$. Here $p_i$ is the projection onto the i-th summand in $\bigoplus_{i=1}^m \left(%
%TCIMACRO{ \dprod _{n\in\mathbb{N}}^{\alpha}}%
%BeginExpansion
{\displaystyle \prod_{n\in \mathbb{N}}^{\alpha}}
%EndExpansion
\left(  P_{i, n}\mathcal{A}_{n}P_{i, n}, \; \rho_{n}\right)\right).$
Hence for 
$\psi_{n} = \oplus_{i=1}^m \psi_{n, i}: \mathcal{B }\to \sum_{i=1}^{m} P_{i, n}%
\mathcal{A}_{n}P_{i, n}$, we have
\begin{equation}
\label{3tensor}\{ \psi_{n}(b)\}_{\alpha} = f(b\otimes1),
\end{equation}
for any $b\in \mathcal{B}$. Define a $\ast$-homomorphism $\delta_{n}: C(X)
\to \sum_{i=1}^{m} P_{i, n}\mathcal{A}_{n}P_{i, n}$ by
\begin{equation}
\label{4tensor}\delta_{n}(\phi) = \sum_{i=1}^{m} f(1\otimes \phi)(w_{i})P_{i,
n}.
\end{equation}
Define $\tilde f_{n}: \mathcal{B }\otimes C(X) \to \mathcal{A}_{n}$ by
\[
\tilde f_{n}(b\otimes \phi) = \psi_{n}(b)\delta_{n}(\phi).
\]
Since $\delta_{n}(C(X))$ and $\psi_{n}(\mathcal{B})$ commute, $\tilde f_{n}$
is a $\ast$-homomorphism. By (\ref{1tensor}), (\ref{2tensor}), (\ref{3tensor}%
), (\ref{4tensor}), for each $j\le k$ we have
\begin{multline}
\left \|  \{ \tilde f_{n}(x_{j})\}_{\alpha} - f(x_{j})\right \|  _{2, \rho} \le2
\epsilon+ \left \|  \left \{  \tilde f_{n}\left(  \sum_{i=1}^{N_{j}} b_{i}%
^{(j)}\otimes \phi_{i}^{(j)}\right)  \right \}  _{\alpha} - f\left(  \sum
_{i=1}^{N_{j}}b_{i}^{(j)}\otimes \phi_{i}^{(j)}\right)  \right \|  _{2, \rho}\\
= 2\epsilon+ \left \|  \sum_{i=1}^{N_{j}} \{ \psi_{n}(b_{i}^{(j)})\}_{\alpha}\{
\delta_{n}(\phi_{i}^{(j)}\}_{\alpha} - \sum_{i=1}^{N_{j}} f(b_{i}^{(j)}%
\otimes1) f(1\otimes \phi_{i}^{(j)})\right \|  _{2, \rho}\\
= 2\epsilon+ \left \|  \sum_{i=1}^{N_{j}} f(b_{i}^{(j)}\otimes1) \left(
\sum_{l=1}^{m} f(1\otimes \phi_{i}^{(j)})(w_{l})\chi_{E_{l}} - f(1\otimes
\phi_{i}^{(j)})\right)  \right \|  _{2, \rho}\\
\le2\epsilon+ N_{j}M\frac{\epsilon}{NM}\le3 \epsilon.
\end{multline}

\end{proof}

\begin{proposition}
Suppose the class $\mathcal{C} \subseteq RR0$ is closed under taking finite direct sums. Then the class of
$\mathcal{C}$-tracially stable C*-algebras is closed under taking finite direct sums.
\end{proposition}

\begin{proof}
We will show it for direct sums with 2 summands, and the general case is
similar. Let $\mathcal{A}, \mathcal{B}$ be $\mathcal{C}$-tracially stable,
$\pi: \mathcal{A }\oplus \mathcal{B}\to%
%TCIMACRO{\dprod _{n\in\mathbb{N}}^{\alpha}}%
%BeginExpansion
{\displaystyle \prod_{n\in \mathbb{N}}^{\alpha}}
%EndExpansion
\left(  \mathcal{D}_{n},\rho_{n}\right)  $ a unital $\ast$-homomorphism and
all ${\mathcal{D}}_{n}\in \mathcal{C}.$ Let
\[
p = \pi(1_{\mathcal{A}}), \; q = \pi(1_{\mathcal{B}}).
\]
Then $p+q = 1$. By Theorem \ref{commutative} we can find projections $P_{n},
Q_{n} \in{\mathcal{D}}_{n}$ such that
\[
\{P_{n}\}_{\alpha} = p,\; \{Q_{n}\}_{\alpha} = q,\; P_{n}+Q_{n} =
1_{{\mathcal{D}}_{n}}.
\]
Then
\[
\pi(\mathcal{A}) \subset%
%TCIMACRO{\dprod _{n\in\mathbb{N}}^{\alpha}}%
%BeginExpansion
{\displaystyle \prod_{n\in \mathbb{N}}^{\alpha}}
%EndExpansion
\left(  P_{n} \mathcal{D}_{n}P_{n},\rho_{n}\right) , \; \pi(\mathcal{B})
\subset%
%TCIMACRO{\dprod _{n\in\mathbb{N}}^{\alpha}}%
%BeginExpansion
{\displaystyle \prod_{n\in \mathbb{N}}^{\alpha}}
%EndExpansion
\left(  Q_{n} \mathcal{D}_{n}Q_{n},\rho_{n}\right) .
\]
Since $\mathcal{A}$ and $\mathcal{B}$ are $\mathcal{C}$-tracially stable,
there are unital $\ast$-homomorphisms $\phi: \mathcal{A }\to \prod
P_{n}\mathcal{D}_{n}P_{n}$ and $\psi: \mathcal{B }\to \prod Q_{n}%
\mathcal{D}_{n}Q_{n}$ such that
\[
\{ \phi(a)_{n}\}_{\alpha} = \pi(a), \{ \psi(b)_{n}\}_{\alpha} = \pi(b),
\]
for all $a\in \mathcal{A}, b\in \mathcal{B}$. Then $\tilde \pi$, defined by
$\tilde \pi(a, b) = \phi(a) + \psi(b)$, is an approximate lift of $\pi$.
\end{proof}

\medskip The following proposition is obvious.

\begin{proposition}
\label{freeproducts} The class of $\mathcal{C}$-tracially stable C*-algebras
is closed under finite free products in the unital category for any
$\mathcal{C}$.
\end{proposition}

\noindent {\bf Remark.}
\textrm{The preceding proposition cannot be extended to countable free
products in the unital category. For example, $\mathcal{B}_{n}=\mathcal{M}%
_{n}\left(  \mathbb{C}\right)  \oplus \mathcal{M}_{n+1}\left(  \mathbb{C}%
\right)  $ is matricially tracially stable for each positive integer $n$ by
Lemma \ref{GCD} and Corollary \ref{GCR} in the next section. Let $\alpha$ be a
nontrivial ultrafilter containing the set $\{k! \;|\; k\in \mathbb{N}\}$. Then,
for each $n\in \mathbb{N}$, $\mathcal{B}_{n}$ embeds unitally into $%
%TCIMACRO{\dprod _{n\in\mathbb{N}}^{\alpha}}%
%BeginExpansion
{\displaystyle \prod_{n\in \mathbb{N}}^{\alpha}}
%EndExpansion
\left(  \mathcal{M}_{i},\tau_{i}\right) $ and hence there is a unital $\ast
$-homomorphism from the unital free product $\ast_{n\in \mathbb{N}}%
\mathcal{B}_{n}$ into $%
%TCIMACRO{\dprod _{n\in\mathbb{N}}^{\alpha}}%
%BeginExpansion
{\displaystyle \prod_{n\in \mathbb{N}}^{\alpha}}
%EndExpansion
\left(  \mathcal{M}_{i},\tau_{i}\right) $. However, $\ast_{n\in \mathbb{N}%
}\mathcal{B}_{n}$ has no finite-dimensional representations. Hence $\ast
_{n\in \mathbb{N}}\mathcal{B}_{n}$ is not matricially tracially stable.
However, if $\mathcal{A}_{n}$ is a separable tracially-$\mathcal{C}$-stable
$C^{*}$-algebra for each $n\in \mathbb{N}$, then the unitization of the free
product of the $\mathcal{A}_{n}$'s in the nonunital category is
tracially-$\mathcal{C}$-stable. }

\section{Matricial tracial stability}

Notation: $\tau_{n}$ usual trace on $\mathcal{M}_{n}\left(  \mathbb{C}\right)
$. The tracial ultraproduct is denoted
\[%
%TCIMACRO{\dprod _{n\in\mathbb{N}}^{\alpha}}%
%BeginExpansion
{\displaystyle \prod_{n\in \mathbb{N}}^{\alpha}}
%EndExpansion
\left(  \mathcal{M}_{n}\left(  \mathbb{C}\right)  ,\tau_{n}\right)  .
\]

\subsection{Embeddable traces}

\textbf{Definition }Suppose $\mathcal{A}$ is a unital separable C*-algebra and
$\tau$ is a tracial state on $\mathcal{A}$. We say $\tau$ is \emph{embeddable}
if there is a non-trivial ultrafilter $\alpha$ on $\mathbb{N}$ and a unital
$\ast$-homomorphism%
\[
\pi:\mathcal{A}\rightarrow%
%TCIMACRO{\dprod _{n\in\mathbb{N}}^{\alpha}}%
%BeginExpansion
{\displaystyle \prod_{n\in \mathbb{N}}^{\alpha}}
%EndExpansion
\left(  \mathcal{M}_{n}\left(  \mathbb{C}\right)  ,\tau_{n}\right)
\]
such that%
\[
\tau_{\alpha}\circ \pi=\tau.
\]

If Connes' embedding theorem holds, then every tracial state is embeddable.

A tracial state $\tau$ is \emph{finite-dimensional} if there is a
finite-dimensional C*-algebra $\mathcal{B}$ with a tracial state
$\tau_{\mathcal{B}}$ and a unital *-homomorphism
\[
\pi:\mathcal{A}\rightarrow \mathcal{B}%
\]
such that%
\[
\tau_{\mathcal{B}}\circ \pi=\tau.
\]
A tracial state $\tau$ is called \emph{matricial} if there is a positive
integer $n$ and a unital *-homomorphism%
\[
\pi:\mathcal{A}\rightarrow \mathcal{M}_{n}\left(  \mathbb{C}\right)
\]
such that%
\[
\tau=\tau_{n}\circ \pi.
\]

\bigskip

We say that a matricial tracial state is a \emph{factor matricial state} if
the $\pi$ above can be chosen to be surjective.

\bigskip

\begin{lemma}
\label{basic} A $C^{\ast}$-algebra $A=C^{\ast}\left(  D\right)  $ with a
tracial state $\tau$ admits a trace-preserving $\ast$-homomorphism into an ultraproduct of matrix algebras
if and only if, for every $\epsilon>0$, every finite tuple $(a_{1},...,a_{n})$
of elements in $D$, and every finite set $F$ of $\ast$-monomials, there is a
positive integer $k$ and a tuple $(A_{1},...,A_{n})$ of $k\times k$ matrices
such that

\begin{enumerate}
\item $\left \Vert A_{j}\right \Vert \leq \left \Vert a_{j}\right \Vert +1$ for
$1\leq j\leq n$,

\item $|\tau(m(a_{1},...,a_{n}))-\tau_{k}(m(A_{1},...,A_{n})|<\epsilon$ for
every $m\in F$.
\end{enumerate}
\end{lemma}

\begin{proof}
The "only if" part is obvious. For the other direction, let $\Lambda$ be the
set of all triples $\lambda=\left(  \varepsilon_{\lambda},N_{\lambda
},E_{\lambda}\right)  $ with $\varepsilon_{\lambda}>0,$ $E_{\lambda}$ a finite
subset of $D$, $N_{\lambda}$ a positive integer, partially ordered by$\left(
\geq,\leq,\subseteq \right)  .$ By the hypothesis, for each $\lambda \in \Lambda$
there is a positive integer $k_{\lambda}$ and there is a function $f_{\lambda
}:D\rightarrow \mathcal{M}_{k_{\lambda}}$ with $f_{\lambda}\left(  1\right)
=1$ and such that,

\begin{enumerate}
\item[i.] $\left \Vert f_{\lambda}\left(  a\right)  \right \Vert \leq \left \Vert
a\right \Vert +1$ for every $a\in A,$  

\item[ii.] For all monomials $m$ with degree at most $N_{\lambda}$, and all
$a_{1},\ldots,a_{n}\in E_{\lambda}$,%
\[
|\tau(m(a_{1},...,a_{n}))-\tau_{k}(m(f_{\lambda}(a_{1}),..., f_{\lambda}(a_{n}))|<\varepsilon_{\lambda}.
\]
(We can define $f\left(  a\right)  =0$ on $A\backslash E_{\lambda}$.)
\end{enumerate}

Define $F:D\rightarrow%
%TCIMACRO{\dprod _{\lambda\in\Lambda}}%
%BeginExpansion
{\displaystyle \prod_{\lambda \in \Lambda}}
%EndExpansion
\left(  \mathcal{M}_{k_{\lambda}},\tau_{k_{\lambda}}\right)  $ by $F\left(
a\right)  \left(  \lambda \right)  =f_{\lambda}\left(  a\right)  .$ Let
$\alpha$ be an ultrafilter on $\Lambda$ containing $\left \{  \lambda
:\lambda \geq \lambda_{0}\right \}  $ for each $\lambda_{0}\in \Lambda$. Define
\[
\rho:A\rightarrow%
%TCIMACRO{\dprod _{\lambda\in\Lambda}^{\alpha}}%
%BeginExpansion
{\displaystyle \prod_{\lambda \in \Lambda}^{\alpha}}
%EndExpansion
\left(  \mathcal{M}_{k_{\lambda}},\tau_{k_{\lambda}}\right)
\]
by $\rho \left(  a\right)  =\left[  F\left(  a\right)  \right]  _{\alpha}$. Let
$A_{0}$ denote the unital $\ast$-algebra generated by $D$. It follows from the
definition of the $f_{\lambda}$'s that $\rho$ extends to a unital $\ast
$-homomorphism $\pi_{0}:A_{0}\rightarrow%
%TCIMACRO{\dprod _{\lambda\in\Lambda}^{\alpha}}%
%BeginExpansion
{\displaystyle \prod_{\lambda \in \Lambda}^{\alpha}}
%EndExpansion
\left(  \mathcal{M}_{k_{\lambda}},\tau_{k_{\lambda}}\right)  $ with
$\tau_{\alpha}\circ \pi_{0}=\tau$. We can now uniquely extend, by continuity,
$\pi_{0}$ to a unital $\ast$-homomorphism on $A$ such that $\tau=\tau_{\alpha
}\circ \pi$.
\end{proof}

\begin{corollary}\label{AnyUltrafilter}
 If a separable $C^{\ast}$-algebra with a tracial state
admits a trace-preserving $\ast$-homomorphism into $\prod_{n\in \mathbb{N}}^{\alpha}(M_{n},\tau_{n})$, for
some non-trivial ultrafilter $\alpha$, then it admits a trace-preserving $\ast$-homomorphism into
$\prod_{n\in \mathbb{N}}^{\beta}(M_{n},\tau_{n})$, for any non-trivial
ultrafilter $\beta$.
\end{corollary}

\begin{proof}
Suppose $A=C^{\ast}\left(  a_{1},a_{2},\ldots \right)  $ with a tracial state
$\tau$, and suppose $\rho:A\rightarrow \prod_{n\in \mathbb{N}}^{\alpha}%
(M_{n},\tau_{n})$ is a unital $\ast$-homomorphism such that $\tau_{\alpha
}\circ \rho=\tau$. We can assume $\left \Vert a_{n}\right \Vert =1$ for all $n$,
and we can let $\rho \left(  a_{n}\right)  =\left \{  b_{n,\kappa}\right \}
_{\alpha}$ with $\left \Vert b_{n,\kappa}\right \Vert \leq \left \Vert
a_{n}\right \Vert =1$ for all $k\geq1$. For each positive integer $s$ there is
a positive $k_{s}$ such that, for every $\ast$-monomial $m$ with degree at
most $s,$
\[
\left \vert \tau \left(  m\left(  a_{1},\ldots,a_{s}\right)  \right)
-\tau_{k_{s}}\left(  m\left(  b_{1,k_{s}},\ldots,b_{s,k_{s}}\right)  \right)
\right \vert <\frac{1}{2s}.
\]
Suppose $t>2sk_{s}$. Then by dividing $t$ by $k_{s}$ we get $t=k_{s}q+r$ with
$0\leq r<k_{s}$. We define $c_{j}=b_{j,k_{s}}^{\left(  q\right)  }\oplus0_{r}$
(where $b^{\left(  q\right)  }$ is a direct sum of $q$ copies of $b$ and
$0_{r}$ is an $r\times r$ zero matrix. A simple computation gives, for any
monomial $m$ of degree at most $s$ that%
\[
\tau_{t}\left(  m\left(  c_{1},\ldots,c_{s}\right)  \right)  =\tau_{k_{s}%
}\left(  m\left(  b_{1,k_{s}},\ldots,b_{s,k_{s}}\right)  \right)  \left(
1-\frac{r}{t}\right)  .
\]
Since $r/t<1/2s$ and $\left \vert \tau_{k_{s}}\left(  m\left(  b_{1,k_{s}%
},\ldots,b_{s,k_{s}}\right)  \right)  \right \vert \leq1$, we see that%
\[
\left \vert \tau \left(  m\left(  a_{1},\ldots,a_{s}\right)  \right)  -\tau
_{t}\left(  m\left(  c_{1},\ldots,c_{s}\right)  \right)  \right \vert <\frac
{1}{s}.
\]
Hence we get representatives like $\left(  c_{1},\ldots,c_{s}\right)  $ for
all $t>2sk_{s}$. From this we see that there is a function $F:\left \{
a_{1},a_{2},\ldots \right \}  \rightarrow$ $\prod_{n\in \mathbb{N}}(M_{n}%
,\tau_{n})$ so that if $f\left(  a_{n}\right)  =\left(  d_{n,1},d_{n,2}%
,\ldots \right)  $, we have%
\[
\lim_{n\rightarrow \infty}\tau_{n}\left(  m\left(  d_{n,1},d_{n,2}%
,\ldots \right)  \right)  =\tau \left(  m\left(  a_{1},a_{2},\ldots \right)
\right)
\]
for any $\ast$-monomial $m$. Hence, for any free ultrafilter $\beta$,
\[
\lim_{n\rightarrow \beta}\tau_{n}\left(  m\left(  d_{n,1},d_{n,2}%
,\ldots \right)  \right)  =\tau \left(  m\left(  a_{1},a_{2},\ldots \right)
\right)  .
\]
Thus the map $\pi_{0}:\left \{  a_{1},a_{2},\ldots \right \}  \rightarrow$
$\prod_{n\in \mathbb{N}}^{\beta}(M_{n},\tau_{n})$ defined by%
\[
\pi_0 \left(  a_{k}\right)  =F\left(  a_{k}\right)  _{\beta}%
\]
extends to a unital $\ast$-homomorphism $\pi:A\rightarrow$ $\prod
_{n\in \mathbb{N}}^{\beta}(M_{n},\tau_{n})$ such that $\tau_{\beta}\circ
\pi=\tau$.
\end{proof}

It follows easily from Lemma \ref{basic} that the set of embeddable tracial
states on $\mathcal{A}$ is $\ast$-weak-compact. It is not hard to show that it
is also convex, and contains the finite-dimensional tracial states. The set of
finite-dimensional tracial states is the convex hull of the set of factor
matricial states. Moreover, the set of finite-dimensional tracial states is
contained in the weak*-closure of the set of matricial tracial states.

\subsection{Matricial tracial stability -- necessary conditions}

Let $\mathcal{A}$ be a unital $C^{*}$-algebra and let $\mathcal{J}_{et}\left(
\mathcal{A}\right)  $ be the largest ideal of $\mathcal{A}$ that is
annihilated by every  embeddable trace on $\mathcal{A}$. If $\mathcal{A}$ has
no embeddable traces, then $\mathcal{A}/\mathcal{J}_{et}=\left \{  0\right \}
.$ More generally, the embeddable tracial states of $\mathcal{A}%
/\mathcal{J}_{et}$ separate the points of $\mathcal{A}/\mathcal{J}_{et}$, and
if $\mathcal{A}$ is separable, $\mathcal{A}/\mathcal{J}_{et}$ always has a
faithful tracial embeddable state. Clearly, $\mathcal{A}$ is matricially
tracially stable iff $\mathcal{A}/\mathcal{J}_{et}$ is matricially tracially
stable. $C^{*}$-algebras without any embeddable tracial states are
automatically matricially tracially stable.

\begin{theorem}
\label{conditions} \label{mat nec}Suppose $\mathcal{A}$ is a separable unital
matricially tracially stable $C^{*}$-algebra with at least one embeddable
trace. Then

\begin{enumerate}
\item For all but finitely many positive integers $n$ there is a unital
representation $\pi_{n}:\mathcal{A}\rightarrow \mathcal{M}_{n}\left(
\mathbb{C}\right)  $

\item The set of embeddable tracial states on $\mathcal{A}$ is the $\ast
$-weak-closed convex hull of the set of matricial tracial states.
\end{enumerate}
\end{theorem}

\begin{proof}
$\left(  1\right)  .$ Suppose, via contradiction, that the set $E$ of positive
integers $n$ for which there is no $n$-dimensional representation of
$\mathcal{A}$ is infinite. Let $\alpha$ be a non-trivial ultrafilter on
$\mathbb{N}$ with $E\in \alpha$. Let $\tau$ be the limit trace on $%
%TCIMACRO{\dprod _{n\in\mathbb{N}}^{\alpha}}%
%BeginExpansion
{\displaystyle \prod_{n\in \mathbb{N}}^{\alpha}}
%EndExpansion
\left(  \mathcal{M}_{n}\left(  \mathbb{C}\right)  ,\tau_{n}\right)  $. Then,
if $\rho$ is an embeddable trace on $\mathcal{A}$, by Corollary
\ref{AnyUltrafilter} there is a representation $\pi:\mathcal{A}\rightarrow%
%TCIMACRO{\dprod _{n\in\mathbb{N}}^{\alpha}}%
%BeginExpansion
{\displaystyle \prod_{n\in \mathbb{N}}^{\alpha}}
%EndExpansion
\left(  \mathcal{M}_{n}\left(  \mathbb{C}\right)  ,\tau_{n}\right)  $ such
that $\tau \circ \pi=\rho$. Since there $\mathcal{A}$ is matricially stable, the
set $\mathbb{N}\backslash E\in \alpha,$ which is a contradiction.

$\left(  2\right)  .$ If $\rho$ is an embeddable trace and $\alpha$ is a
non-trivial ultrafilter on $\mathbb{N}$ and $\pi:\mathcal{A}\rightarrow%
%TCIMACRO{\dprod _{n\in\mathbb{N}}^{\alpha}}%
%BeginExpansion
{\displaystyle \prod_{n\in \mathbb{N}}^{\alpha}}
%EndExpansion
\left(  \mathcal{M}_{n}\left(  \mathbb{C}\right)  ,\tau_{n}\right)  $ is a
representation such that $\tau \circ \pi=\rho$, then the matricial tracial
stability of $\mathcal{A}$ implies there is a set $F\in \alpha$ and, for each
$n\in F$, there is a representation $\pi_{n}:A\rightarrow \mathcal{M}%
_{n}\left(  \mathbb{C}\right)  $ such that the $\ast$-weak-limit along
$\alpha$ of the finite-dimensional traces $\tau_{n}\circ \pi_{n}$.
\end{proof}

\begin{corollary}
Suppose $\mathcal{A}$ is separable and matricially tracially stable. Then
$\mathcal{A}/\mathcal{J}_{et}$ is RFD. If the collection of embeddable traces
on $\mathcal{A}$ separates points of $\mathcal A$, then $\mathcal{A}$ is RFD.
\end{corollary}

The set of positive integers $n$ for which there is an $n$-dimensional
representation is an additive semigroup (direct sums) that is generated by the
 $k\in \mathbb{N}$ for which there is a $k$-dimensional irreducible
representation. The following simple lemma, which yields a reformulation of the
first condition in Theorem \ref{conditions}, should be well-known.

\begin{lemma}
\label{GCD} Suppose $n_{1}<n_{2}<\cdots<n_{s}$. The following are equivalent:

\begin{enumerate}
\item $GCD\left(  n_{1},\ldots,n_{k}\right)  =1.$

\item There is a positive integer $N$ such that every integer $n\geq N$ can be
written as
\[
n=\sum_{k=1}^{s}a_{k}n_{k}%
\]
with $a_{1},\ldots a_{k}$ nonnegative integers.
\end{enumerate}
\end{lemma}

\begin{proof}
$\left(  2\right)  \Longrightarrow \left(  1\right)  .$ If we write
$N=\sum_{k=1}^{s}a_{k}n_{k}$ and $N+1=\sum_{k=1}^{s}b_{k}n_{k},$ then%
\[
\sum_{k=1}^{s}\left(  b_{k}-a_{k}\right)  n_{k}=1,
\]
which implies $\left(  1\right)  .$

$\left(  1\right)  \Longrightarrow \left(  2\right)  .$ If $GCD\left(
n_{1},\ldots,n_{k}\right)  =1$, there are integers $s_{1},\ldots,s_{k}$ such
that
\[
s_{1}n_{1}+\cdots+s_{k}n_{k}=1.
\]
Let $m=1+n_{1}\max \left(  \left \vert s_{1}\right \vert ,\ldots,\left \vert
s_{k}\right \vert \right)  $. Let $N=mn_{1}+\cdots+mn_{k}$. Suppose $n\geq$
$N$. Using the division algorithm we can find integers $q\geq0$ and $r$ with
$0\leq r<n_{1}$ such that%
\[
n-N=n_{1}q+r.
\]
Thus%
\[
n=mn_{1}+\cdots+mn_{k}+qn_{1}+r\left(  s_{1}n_{1}+\cdots+s_{k}n_{k}\right)  =
\]%
\[
\left(  m+rs_{1}+q\right)  n_{1}+\sum_{j=2}^{k}\left(  m+rs_{j}\right)
n_{j}.
\]
However, $m+rs_{j}\geq m-n_{1}\left \vert s_{j}\right \vert \geq1$ for $1\leq
j\leq k.$
\end{proof}

\begin{corollary}\label{Condition1Equivalent}
A unital C*-algebra satisfies condition $\left(  1\right)  $ in Theorem
\ref{conditions} if and only if $\left \{  n\in \mathbb{N}:\mathcal{A}\text{ has
an irreducible }n\text{-dimensional representation}\right \}  $ has greatest
common divisor equal to $1$.
\end{corollary}

\begin{lemma}
\label{conditionsequivalent} \label{exist}A separable unital $C^{*}$-algebra
$\mathcal{A}$ satisfies statements $\left(  1\right)  $ and $\left(  2\right)
$ in Theorem \ref{mat nec} if and only if, for every embeddable tracial state
$\tau$ on $\mathcal{A}$, there is a positive integer $n_{0}$ and, for each
$n\geq n_{0}$ there is a unital $\ast$-homomorphism $\rho_{n}%
:\mathcal{A\rightarrow M}_{n}\left(  \mathbb{C}\right)  $ such that, for every
$a\in \mathcal{A}$,%
\[
\tau \left(  a\right)  =\lim_{n\rightarrow \infty}\tau_{n}\left(  \rho
_{n}\left(  a\right)  \right)
\]

\end{lemma}

\begin{proof}
The "if" part is clear. Suppose statements $\left(  1\right)  $ and $\left(
2\right)  $ in Theorem \ref{mat nec} are true. Let $\left \{  a_{1}%
,a_{2},\ldots \right \}  $ be a norm dense subset if the unit ball of
$\mathcal{A}$. Let $\tau$ be an embeddable tracial state on $\mathcal{A}$, and
suppose $N$ is a positive integer. Then there is a convex combination $\sigma
$, with rational coefficients, of matricial states, and hence even of factor
matricial states, such that%
\[
\left \vert \tau \left(  a_{j}\right)  -\sigma \left(  a_{j}\right)  \right \vert
<1/2N
\]
for $1\leq j\leq N$. Thus there exist a positive integer $m$ and positive
integers $s_{1},\ldots,s_{\nu}$ with $s_{1}+\cdots+s_{\nu}=m$ and positive
integers $t_{1},\ldots,t_{\nu}$ with surjective representations $\pi
_{i}:\mathcal{A}\rightarrow \mathcal{M}_{t_{i}}\left(  \mathbb{C}\right)  $
such that%
\[
\sigma=\sum_{i=1}^{\nu}\frac{s_{i}}{m}\tau_{t_{i}}\circ \pi_{i}.
\]
Let $M\in \mathbb N$ be such $\frac{s_i}{t_i}M$ is an integer, for all $i \le \nu$.
If $\rho^{\left(  k\right)  }$ denotes a direct sum of $k$ copies of a
representation $\rho,$ then let
\[
\pi=\oplus_{i}\pi_{i}^{({\frac{s_i}{t_i}M})}:\mathcal{A}\rightarrow \mathcal{M}_{mM}\left(
\mathbb{C}\right)  .
\]
It is easy to check that $\sigma=\tau_{mM}\circ \pi$. Since by the assumption the statement (1) in
Theorem \ref{mat nec} holds, there is  $n_{0}$ such that for any $n\ge n_{0}$
there is a unital representation into $\mathcal{M}_{n}(\mathbb{C})$. Since for
a positive integer $q$ we have $\sigma=\tau_{qmM}\circ \pi^{\left(  q\right)  }%
$, we can assume that $mM>n_{0}$. Suppose $n>2mM$. We can write $n=amM+b$ with
$a\geq2$ and $0\leq b<mM$. Since $n_{0}\leq mM+b$, there is a unital
representation $\rho:\mathcal{A}\rightarrow \mathcal{M}_{mM+b}\left(
\mathbb{C}\right)  $. Let
\[
\pi_{N,n}=\pi^{ \left(  a-1\right)    }\oplus \rho
:\mathcal{A}\rightarrow \mathcal{M}_{n}\left(  \mathbb{C}\right)  .
\]

Then
\begin{multline*}
\tau_{n} \circ \pi_{N, n}(a_{k}) = \frac{(a-1)mM \tau_{Mm}\circ
\pi(a_{k}) + (mM+b) \tau_{mM+b} \circ \rho(a_{k})}{n} =
\\\frac{
(a-1)mM}{n}\sigma(a_{k}) + \frac{mM+b}{n}\tau_{m+b}\circ \rho(a_{k})
 =\\
(1- \frac{mM+b}{n})\sigma(a_{k}) + \frac{mM+b}{n}\tau_{m+b}\circ \rho(a_{k}).
\end{multline*}

There is a positive integer $k_{N}$ such that if $n\geq k_{N},$ then
$\frac{mM+b}{n}< \frac{1}{4N}$. Then for all $n\ge k_{N}$%

\[
\left \vert \tau_{n}\circ \pi_{N,n}\left(  a_{k}\right)  -\sigma \left(
a_{k}\right)  \right \vert <1/2N
\]
for $1\leq k\leq N$ and hence
\[
\left \vert \tau \left(  a_{k}\right)  -\tau_{n}\circ \pi_{N,n}\left(
a_{k}\right)  \right \vert <1/N
\]
for $1\leq k\leq N$ . We can easily arrange $k_{1}<k_{2}<\cdots$ and we can
define%
\[
\rho_{n}=\pi_{N,n}\text{ if }k_{N}\leq n<k_{N+1}.
\]
\bigskip
\end{proof}

\subsection{Matricial tracial stability for tracially nuclear C*-algebras}

Recall that a unital $C^{*}$-algebra $\mathcal{A}$ is \emph{nuclear} if and
only if, for every Hilbert space $H$ and every unital $\ast$-homomorphism
$\pi:\mathcal{A}\rightarrow B\left(  H\right)  $ the von Neumann algebra
$\pi \left(  \mathcal{A}\right)  ^{\prime \prime}$ is hyperfinite. The algebra
$\mathcal{A}$ is \emph{tracially nuclear} \cite{Hadwin-Li} if, for every
tracial state $\tau$ on $\mathcal{A}$, the algebra $\pi_{\tau}\left(
\mathcal{A}\right)  ^{\prime \prime}$ is hyperfinite, where $\pi_{\tau}$ is the
GNS representation for $\tau$. It is easy to show that if $\mathcal{A}$ is
tracially nuclear, then every trace on $\mathcal{A}$ is embeddable.

\begin{theorem}
\label{nuclear} Suppose $\mathcal{A}$ is a unital separable tracially nuclear
C*-algebra with at least one tracial state. The following are equivalent

\begin{enumerate}
\item $\mathcal{A}$ is matricially tracially stable

\item $\mathcal{A}$ satisfies the conditions $\left(  1\right)  $ and $\left(
2\right)  $ in Theorem \ref{mat nec}.

\item $\mathcal{A}$ is W*-factor tracially stable.
\end{enumerate}
\end{theorem}

\begin{proof}
$1)\Rightarrow2)$ This follows from Theorem \ref{mat nec}.

$2)\Rightarrow3)$. Suppose $\mathcal{A=C}^{\ast}\left(  x_{1},x_{2}%
,\ldots \right)  $, with all $x_i$'s being contractions,  suppose $\alpha$ is a free ultrafilter on $\mathbb{N}$, and
suppose $\sigma:\mathcal{A}\rightarrow%
%TCIMACRO{\dprod _{n\in\mathbb{N}}^{\alpha}}%
%BeginExpansion
{\displaystyle \prod_{n\in \mathbb{N}}^{\alpha}}
%EndExpansion
(\mathcal{M}_{n},\rho_{n})$, where $(\mathcal{M}_{n},\rho_{n})$ is a factor
von Neumann algebra with trace $\rho_{n}$. Denote by $\rho_{\alpha}$ the limit
trace on $%
%TCIMACRO{\dprod _{n\in\mathbb{N}}^{\alpha}}%
%BeginExpansion
{\displaystyle \prod_{n\in \mathbb{N}}^{\alpha}}
%EndExpansion
(\mathcal{M}_{n},\rho_{n})$. Since $\rho_{\alpha}$ a faithful trace on $%
%TCIMACRO{\dprod _{n\in\mathbb{N}}^{\alpha}}%
%BeginExpansion
{\displaystyle \prod_{n\in \mathbb{N}}^{\alpha}}
%EndExpansion
\left(  \mathcal{M}_{n},\rho_{n}\right)  $ and since $\mathcal{A}$ is
tracially nuclear, it follows that $\sigma \left(  \mathcal{A}\right)
^{\prime \prime}$ is hyperfinite.

Since $\alpha$ is an ultrafilter, either $\left \{  n\in \mathbb{N}%
:\dim \mathcal{M}_{n}<\infty \right \}  \in \alpha$ or $\left \{  n\in
\mathbb{N}:\dim \mathcal{M}_{n}=\infty \right \}  \in \alpha.$ We first consider
the latter case. In this case we can assume that every $\mathcal{M}_{n}$ is a
$II_{1}$-factor. In this case, it follows from \cite{Hadwin-Li} that, for each
$n\in \mathbb{N}$ there is a unital $\ast$-homomorphism $\gamma_{n}%
:\sigma \left(  \mathcal{A}\right)  ^{\prime \prime}\rightarrow \mathcal{M}_{n}$
such that, for every $b\in \sigma \left(  \mathcal{A}\right)  ^{\prime \prime}$
\[
\lim_{n\to \alpha}\rho_{n}\left(  \gamma_{n}\left(  b\right)  \right)  =\rho_{\alpha}\left(
b\right)  ,
\]
and%
\[
b=\left \{  \gamma_{n}\left(  b\right)  \right \}  _{\alpha}.
\]
Hence $\left \{  \gamma_{n}\circ \sigma \right \}_{n\in \mathbb N}  $ is an approximate lifting of
$\sigma.$

Next suppose $E:= \left \{  n\in \mathbb{N}%
:\dim \mathcal{M}_{n}<\infty \right \}  \in \alpha$ and $\mathcal{M}_{n}=\mathcal{M}_{k_n}\left(  \mathbb{C}\right)  $ and
$\rho_{n}=\tau_{k_n}$ for each $n\in E$.

By assumption and Lemma \ref{conditionsequivalent}, there is an $n_{0}%
\in \mathbb{N}$ and, for every $n\geq n_{0}$ there is a unital $\ast
$-homomorphism $\pi_{n}:\mathcal{A}\rightarrow \mathcal{M}_{n}\left(
\mathbb{C}\right)  $ such that, for every $a\in \mathcal{A}$,%
\[
\tau_{\alpha}\left(  a\right)  =\lim_{k\rightarrow \infty}\tau_{n}\left(
\pi_{n}\left(  a\right)  \right)  .
\]
For each $k\in \mathbb{N}$, write $\sigma \left(  x_{k}\right)  =\left \{
x_{k}\left(  n\right)  \right \}  _{\alpha}$. For any $s\in \mathbb{N}$ and any
$\ast$-monomial $m\left(  t_{1},\ldots,t_{s}\right)  $ we have%
\[
\lim_{n\rightarrow \alpha}\tau_{n}\left(  m\left(  x_{1}\left(  n\right)
,\ldots,x_{s}\left(  n\right)  \right)  \right)  =\tau_{\alpha}\left(
m\left(  \sigma \left(  x_{1}\right)  ,\ldots,\sigma \left(  x_{s}\right)
\right)  \right)
\]
and%
\[
\lim_{n\rightarrow \infty}\tau_{n}\left(  m\left(  \pi_{n}\left(  x_{1}\right)
,\ldots,\pi_{n}\left(  x_{s}\right)  \right)  \right)  =\tau_{\alpha}\left(
m\left(  \sigma \left(  x_{1}\right)  ,\ldots,\sigma \left(  x_{s}\right)
\right)  \right)  .
\]
Since $\sigma \left(  \mathcal{A}\right)  ^{\prime \prime}$ is hyperfinite, it
follows from Connes' theorem that, for each $s\in \mathbb{N}$, $\sigma \left(
C^{\ast}\left(  x_{1},\ldots,x_{s}\right)  \right)  ^{\prime \prime}%
\subset \sigma \left(  \mathcal{A}\right)  ^{\prime \prime}$ is also hyperfinite.
It now follows from Theorem \ref{Shanghai} that there is a sequence of
unitaries $U_{n}\in \mathcal{M}_{k_n}\left(  \mathbb{C}\right)  $ such that, for
every $s\in \mathbb{N}$,%
\[
\lim_{n\rightarrow \alpha}\left \Vert x_{s}\left(  n\right)  -U_{n}^{\ast}%
\pi_{n}\left(  x_{s}\right)  U_{n}\right \Vert _{2}=0.
\]
Hence $\left \{  U_{n}^{\ast}\pi_{n}\left(  \cdot \right)  U_{n}\right \}  $ is
an approximate lifting of $\sigma.$

$3)\Rightarrow1)$ is obvious.
\end{proof}

\medskip

Recall that a $C^{*}$-algebra $\mathcal{A}$ is called \textit{GCR} (or
\textit{type I}) if for any its irreducible representation $\pi: \mathcal{A
}\to B(H)$, $\pi(\mathcal{A})$ contains all the compact operators.

\begin{corollary}
\label{GCR} Suppose $\mathcal{A}$ is a separable GCR unital $C^{*}$-algebra
satisfying condition $\left(  1\right)  $ in Theorem \ref{mat nec}. Then
$\mathcal{A}$ is $W^{*}$-factor tracially stable.
\end{corollary}

\begin{proof}
The extreme points of the tracial states are the factor tracial states
(\cite{DonExtreme}). A factor representation of a GCR $C^{*}$-algebra must
yield a factor von Neumann algebra of type I, which must be isomorphic to some
$B(H)$. If it has a trace, then $H$ must be finite-dimensional. Thus the
factorial tracial states are finite-dimensional, and the Krein-Milman theorem
gives that all tracial states are in the weak*-closed convex hull of the
finite-dimensional states. Thus the condition $\left(  2\right)  $ in Theorem
\ref{mat nec} holds.
\end{proof}

\bigskip

The next statement follows from Theorem \ref{nuclear}, Theorem \ref{mat nec}
and Lemma \ref{conditionsequivalent}.

\begin{theorem}
\label{nuclear2} Suppose $\mathcal{A}$ is a separable tracially nuclear
C*-algebra with at least one tracial state. The following are equivalent

\begin{enumerate}
\item $\mathcal{A}$ is matricially tracially stable

\item for every tracial state $\tau$ on $\mathcal{A}$, there is a positive
integer $n_{0}$ and, for each $n\geq n_{0}$ there is a unital $\ast
$-homomorphism $\rho_{n}:\mathcal{A\rightarrow M}_{n}\left(  \mathbb{C}%
\right)  $ such that, for every $a\in \mathcal{A}$,%
\[
\tau \left(  a\right)  =\lim_{n\rightarrow \infty}\tau_{n}\left(  \rho
_{n}\left(  a\right)  \right) .
\]
(here $\tau_{n}$ is the usual tracial state on $\mathcal{M}_{n}(\mathbb{C})$)

\item $\mathcal{A}$ is W*-factor tracially stable.
\end{enumerate}
\end{theorem}

\medskip

Below we give an example of an RFD nuclear $C^{*}$-algebra which has
finite-dimensional irreducible representations of all dimensions but is not
matricially tracially stable.

\begin{example}
\label{example}\textrm{Suppose $0<\theta_{1}<\theta_{2}<,1$ are irrational and
$\left \{  1,\theta_{1},\theta_{2}\right \}  $ is linearly independent over
$\mathbb{Q}$. Let $\lambda_{k}=e^{2\pi i\theta_{k}}$ for $k=1,2$. Let
$\mathcal{A}_{\theta_{k}}$ be the irrational rotation algebra generated by
unitaries $U_{k},V_{k}$ satisfying $U_{k}V_{k}=\lambda_{k}V_{k}U_{k}.$ We know
that each $\mathcal{A}_{\theta_{k}}$ is simple nuclear and has a unique
tracial state $\rho_{k}$. Let $U=U_{1}\oplus U_{2}$ and $V=V_{1}\oplus V_{2}$.
Since $UVU^{\ast}V^{\ast}=\lambda_{1}\oplus \lambda_2,$ we see that $1\oplus0\in
C^{\ast}\left(  U,V\right)  ,$ so
\[
C^{\ast}\left(  U,V\right)  =\mathcal{A}_{\theta_{1}}\oplus \mathcal{A}%
_{\theta_{2}}.
\]
By the linear independence assumption, there is an increasing sequence
$\left \{  n_{k}\right \}  $ of positive integers such that
\begin{equation}
\label{independence}\lambda_{1}^{n_{k}}\rightarrow1\text{ and }\lambda
_{2}^{n_{k}}\rightarrow1
\end{equation}
as $k\rightarrow \infty$. } We can assume also that \begin{equation}\label{ForCondition1} n_1 = 2 \;\text{and} \, n_2 = 3.\end{equation}

\textrm{For each positive integer $n$ and each $\lambda \in \mathbb{C}$, let
$\left \{  e_{1},\ldots,e_{n}\right \}  $ be the standard orthonormal basis for
$\mathbb{C}^{n}$, let $U_{n,\lambda}$ and $V_{n}$ be the matrices defined by%
\[
U_{n,\lambda}e_{j}=\lambda^{j-1}e_{j}%
\]
and%
\[
V_{n}e_{j}=e_{j+1}\text{ for }1\leq j<n;\text{ }V_{n}e_{n}=e_{1}.
\]
It follows from (\ref{independence}) that
\[
\left \Vert U_{n_{k},\lambda_{s}}V_{n_{k}}-\lambda_sV_{n_{k}}U_{n_{k},\lambda_{s}%
}\right \Vert \rightarrow0
\]
as $k\rightarrow \infty$ for $s=1,2$. The simplicity of each $\mathcal{A}%
_{\theta s}$ implies that, for every $\ast$-polynomial $p$ and $s=1,2,$ we
have%
\begin{equation}
\label{MF}\left \Vert p\left(  U_{n_{k},\lambda_{s}},V_{n_{k}}\right)
-p\left(  U_{s},V_{s}\right)  \right \Vert \rightarrow0.
\end{equation}
}

\textrm{The uniqueness of the trace on each $\mathcal{A}_{\theta s}$ implies
that, for every $\ast$-polynomial $p$ and $s=1,2,$ we have}

\textrm{%
\[
\lim_{\alpha}\tau_{n_{k}}\left(  p\left(  U_{n_{k},\lambda_{s}},V_{n_{k}%
}\right)  \right)  = \rho_{s}\left(  p\left(  U_{s},V_{s}\right)  \right)
\]
for any non-trivial ultrafilter $\alpha$, and hence
\begin{equation}
\label{embed}\tau_{n_{k}}\left(  p\left(  U_{n_{k},\lambda_{s}},V_{n_{k}%
}\right)  \right)  \rightarrow \rho_{s}\left(  p\left(  U_{s},V_{s}\right)
\right)
\end{equation}
as $k\rightarrow \infty$. }

\textrm{Let $\hat{U}=\sum_{k\in \mathbb{N}}^{\oplus}U_{n_{k},\lambda_{1}%
}^{\left(  k-1\right)  }\oplus U_{n_{k},\lambda_{2}}$ and $\hat{V}=\sum
_{k\in \mathbb{N}}^{\oplus}V_{n_{k}}^{\left(  k-1\right)  }\oplus V_{n_{k}}$,
and let%
\[
\mathcal{A}=C^{\ast}\left(  \hat{U},\hat{V}\right)  +\mathcal{K},
\]
where
\[
\mathcal{K}=\sum \nolimits_{k\in \mathbb{N}}^{\oplus}{\mathcal{M}}_{kn_{k}}(
\mathbb{C}) .
\]
It follows that $\mathcal{K}$ is nuclear and by (\ref{MF}) $\mathcal{A}%
/\mathcal{K}$ is isomorphic to $C^{\ast}\left(  U,V\right)  =\mathcal{A}%
_{\theta_{1}}\oplus \mathcal{A}_{\theta_{2}}$, which is also nuclear. Hence,
$\mathcal{A}$ is nuclear. Moreover, the only finite-dimensional irreducible
representations of $\mathcal{A}$ are the coordinate representations $\pi_{k}$
for $k\in \mathbb{N}$. However, by (\ref{embed}), for every $\ast$-polynomial
$p$, we have
\begin{multline*}
\lim_{k\rightarrow \infty}\tau_{kn_{k}}(\pi_{k}\left(  p\left(  \hat{U},\hat
{V}\right)  \right)  = \lim_{k\rightarrow \infty} \left(  \frac{k-1}{k}
\tau_{n_{k}} \left(  p(U_{n_{k}, \lambda_{1}}, V_{n_{k}})\right)  + \frac
{1}{k} \tau_{n_{k}} \left(  p(U_{n_{k}, \lambda_{2}}, V_{n_{k}})\right)
\right) \\
= \lim_{k\rightarrow \infty} \tau_{n_{k}} \left(  p(U_{n_{k}, \lambda_{1}},
V_{n_{k}})\right)  = \rho_{1}\left(  p\left(  U_{1},V_{1}\right)  \right)  .
\end{multline*}
Hence the trace $\rho$ on $\mathcal{A}$ that annihilates $\mathcal{K}$ and
sends $p\left(  \hat{U},\hat{V}\right)  $ to $\rho_{2}\left(  p\left(
U_{2},V_{2}\right)  \right)  $ is embeddable and cannot be approximated by
finite-dimensional tracial states. Thus $\mathcal{A}$ is nuclear and RFD and
by (\ref{ForCondition1}) and Corollary \ref{Condition1Equivalent} satisfies condition $\left(  1\right)  $ in Theorem \ref{conditions}, but does
not satisfy condition $\left(  2\right)  $, hence is not matricially tracially
stable. }

\end{example}

%\textbf{Is it true that amalgamated free product of commutative algebras is
%matricially tracially stable?}

%\textbf{Do we know an example of algebra which is matricially tracially
%stable, but not $W^{*}$-factor tracially
%stable???????????????????????????????}

\section{Matricial tracial stability and free entropy dimension}

There is a close relationship between matricial stability and D. Voiculescu's
free entropy dimension \cite{V1}, \cite{V2} (via the free orbit dimension
in Hadwin-Shen \cite{HS}). D. Voiculescu defined his free entropy dimension
$\delta_{0}$ (\cite{V1}, \cite{V2}), and he applied it to show the existence
of a $II_{1}$ factor von Neumann algebra without a Cartan MASA, solving a
longstanding problem.

Suppose $\mathcal{A}$ is a unital C*-algebra with a tracial state $\tau$.
Suppose $x_{1},\ldots,x_{n}$ are elements in $\mathcal{A}$, $\varepsilon
>0,R>\max_{1\leq j\leq n}\|x_{j}\|$ and $N,k\in \mathbb{N}$. Voiculescu
\cite{V1} defines $\Gamma_{R,\tau}\left(  x_{1},\ldots,x_{n};N,k,\varepsilon
\right)  $ to be the set of all $n$-tuples $\left(  A_{1},\ldots,A_{n}\right)  $
of matrices in $\mathcal{M}_{k}\left(  \mathbb{C}\right)  $ with norm at most
$R$ such that%
\[
\left \vert \tau \left(  m\left(  x_{1},\ldots,x_{n}\right)  \right)  -\tau
_{n}\left(  m\left(  A_{1},\ldots,A_{n}\right)  \right)  \right \vert
<\varepsilon
\]
for all $\ast$-monomials $m\left(  t_{1},\ldots,t_{n}\right)  $ with degree at
most $N.$

Connes' embedding problem is equivalent to the assertion that, for every
unital C*-algebra $\mathcal{B}$ with a tracial state $\tau$, every $n$-tuple
$\left(  x_{1},\ldots,x_{n}\right)  $ of selfadjoint contractions in
$\mathcal{B}$, for any N, for any $R>\max_{1\leq j\leq n}\|x_{j}\|$ and for
every $\varepsilon>0$ there is a positive integer $k$ such that $\Gamma
_{R,\tau}\left(  x_{1},\ldots,x_{n};N,k,\varepsilon \right)  \neq \varnothing.$

Let $\mathcal{M}_{k}\left(  \mathbb{C}\right)  ^{n}=\left(  \left(
A_{1},\ldots,A_{n}\right)  :A_{1},\ldots,A_{n}\in \mathcal{M}_{k}\left(
\mathbb{C}\right)  \right)  $ and define $\left \Vert {}\right \Vert _{2}$ on
$\mathcal{M}_{k}\left(  \mathbb{C}\right)  ^{n}$ by%
\[
\left \Vert \left(  A_{1},\ldots,A_{n}\right)  \right \Vert _{2}^{2}=\sum
_{j=1}^{n}\tau_{k}\left(  A_{j}^{\ast}A_{j}\right)  .
\]
If $A = \left(  A_{1},\ldots,A_{n}\right)  \in \mathcal{M}_{k}\left(
\mathbb{C}\right)  ^{n}$ and $U$ is unitary, we define
$$U^*AU = \left( U^*A_1U, \ldots, U^*A_nU \right).$$
If $\omega>0$, we define the {\it $\omega$-orbit ball} of
$A = \left(  A_{1},\ldots,A_{n}\right)  $, denoted by $\mathcal{U}\left(
A_{1},\ldots,A_{n};\omega \right)  $, to be the set of all $B = \left(
B_{1},\ldots,B_{b}\right)  \in \mathcal{M}_{k}\left(  \mathbb{C}\right)  ^{n}$
such that there is a unitary $U\in \mathcal{M}_{k}\left(  \mathbb{C}\right)  $
such that%
\[
\left \Vert   U^{\ast}AU  -  B\right \Vert _{2}<\omega.
\]
If $\mathcal{E}\subset \mathcal{M}_{k}\left(  \mathbb{C}\right)  ^{n}$, we
define the $\omega$-orbit covering number of $\mathcal{E}$, denoted by
$\nu \left(  \mathcal{E},\omega \right)  $, to be the smallest number of
$\omega$-orbit balls that cover $\mathcal{E}$.

Let $\mathcal{A}=$C*$\left(  x_{1},\ldots,x_{n}\right)  $ and for each
positive integer $k,$ let \textrm{Rep}$\left(  \mathcal{A},k\right)
/\backsimeq$ denote the set of all unital $\ast$-homomorphisms from
$\mathcal{A}$ into $\mathcal{M}_{k}\left(  \mathbb{C}\right)  $ modulo unitary equivalence.
If $\pi_{1},\pi_{2}\in$\textrm{Rep}$\left(  \mathcal{A},k\right)
$ with corresponding images $\left[  \pi \right]  ,$ $\left[  \rho \right]  \in$
\textrm{Rep}$\left(  \mathcal{A},k\right)  /\backsimeq$ we define a metric%
\[
d_{k}\left(  \left[  \pi \right]  ,\left[  \rho \right]  \right)  =\min
\left \Vert \left(  \pi \left(  x_{1}\right)  ,\ldots,\pi \left(  x_{n}\right)
\right)  -\left(  U^{\ast}\rho \left(  x_{1}\right)  U,\ldots,U^{\ast}%
\rho \left(  x_{n}\right)  U\right)  \right \Vert_2
\]
as $U$ ranges over all of the $k\times k$ unitary matrices.

For each $0<\omega<1$, we define%
\[
\nu_{d_{k}}\left(  \mathrm{Rep}\left(  \mathcal{A},k\right)  /\backsimeq
,\omega \right)
\]
to be the minimal number of $d_{k}$-balls of radius $\omega$ it takes to cover
\textrm{Rep}$\left(  \mathcal{A},k\right)  /\backsimeq$.

\begin{lemma}
\label{cardrep}Suppose $\mathcal{A}=$C*$\left(  x_{1},\ldots,x_{n}\right)  $
is matricially tracially stable and $\tau$ is an embeddable tracial state on
$\mathcal{A}$. Let $R>\max_{1\leq j\leq n}\Vert x_{j}\Vert$. For each
$0<\omega<1$ there exists an $m_{\omega}\in \mathbb{N}$ such that, for all
integers $k,N\geq m_{\omega}$ and every $0<\varepsilon<1/m_{\omega}$%
\[
\mathrm{Card}\left(  \mathrm{Rep}\left(  \mathcal{A},k\right)  /\backsimeq
\right)  \geq \nu_{d_{k}}\left(  \mathrm{Rep}\left(  \mathcal{A},k\right)
/\backsimeq,\omega/4\right)  \geq \nu \left(  \Gamma_{R,\tau}\left(
x_{1},\ldots,x_{n};N,k,\varepsilon \right)  ,\omega \right)  .
\]

\end{lemma}

\begin{proof}
Let $0<\omega<1$.
It is easy to deduce from the $\epsilon-\delta$-definition of matricial tracial stability that, there is a positive integer
$m_{\omega}$ such that, for every $k\in \mathbb{N}$ and every $N\geq m_{\omega
}$ and $0<\varepsilon<1/m_{\omega}$, we have for each $B=\left(  b_{1}%
,\ldots,b_{n}\right)  \in \Gamma_{R,\tau}\left(  x_{1},\ldots,x_{n}%
;N,k,\varepsilon \right)  $ a representation $\pi_{B}\in \mathrm{Rep}\left(
\mathcal{A},k\right)  $ such that%
\[
\left \Vert B-\left(  \pi_{B}\left(  x_{1}\right)  ,\ldots,\pi_{B}\left(
x_{n}\right)  \right)  \right \Vert _{2}<\omega/4.
\]
Now suppose $N\geq m_{\omega}$ and $0<\varepsilon<1/m_{\omega}$. It follows from the
definition of $s=\nu \left(  \Gamma_{R,\tau}\left(  x_{1},\ldots,x_{n}%
;N,k,\varepsilon \right)  ,\omega \right)  $ that there is a collection
$$\left \{  B_{j}\in \Gamma_{R,\tau}\left(  x_{1},\ldots,x_{n}%
;N,k,\varepsilon \right):1\leq j\leq s\right \}  $$ so that, for any $k\times k$ unitary
$U$ and $1\leq i\neq j\leq j$,%
\[
\left \Vert U^{\ast}B_{i}U-B_{j}\right \Vert _{2}\geq \omega \text{.}%
\]
For each $j,$ $1\leq j\leq s$. It easily follows that, for $1\leq i\neq j\leq
s,$%
\[
d_{k}\left(  \left[  \pi_{B_{i}}\right]  ,\left[  \pi_{B_{j}}\right]  \right)
\geq \omega/2.
\]
Hence every $d_{k}$-ball with radius $\omega/4$ contains at most one of
$\left[  \pi_{B_j}\right]  $, $1\leq j\leq s.$ Hence,%
\[\mathrm{Card}\left(  \mathrm{Rep}\left(  \mathcal{A},k\right)  /\backsimeq
\right)  \geq
\nu_{d_{k}}\left(  \mathrm{Rep}\left(  \mathcal{A},k\right)  /\backsimeq
,\omega/4\right)  \geq \nu \left(  \Gamma_{R,\tau}\left(  x_{1},\ldots
,x_{n};N,k,\varepsilon \right)  ,\omega \right)  .
\]

\end{proof}

\medskip

In \cite{HS} the first named author and J. Shen defined the free-orbit
dimension $\mathfrak{K}_{1}\left(  x_{1},\ldots,x_{n};\tau \right)  $. First let %
\[
\mathfrak{K}\left(  x_{1},\ldots,x_{n};\tau,\omega \right)  =\inf
_{\varepsilon,N}\limsup_{k\rightarrow \infty}\frac{\log \nu \left(
\Gamma_{R,\tau}\left(  x_{1},\ldots,x_{n};N,k,\varepsilon \right)
,\omega \right)  }{k^{2}\left \vert \log \omega \right \vert },
\]
and let %
\[
\mathfrak{K}_{1}\left(  x_{1},\ldots,x_{n};\tau \right)  =\limsup
_{\omega \rightarrow0^{+}}\mathfrak{K}\left(  x_{1},\ldots,x_{n};\tau
,\omega \right)  .
\]
If $\mathcal{A}=$C*$\left(  x_{1},\ldots,x_{n}\right)  $ is matricially
tracially stable and $\tau$ is an embeddable tracial state on $\mathcal A$, then by Lemma \ref{cardrep}%
\[
\frac{1}{\left \vert \log \omega \right \vert }\limsup_{k\rightarrow \infty}%
\frac{\log \mathrm{Card}\left(  \mathrm{Rep}\left(  \mathcal{A},k\right)
/\backsimeq \right)  }{k^{2}}\geq \mathfrak{K}\left(  x_{1},\ldots,x_{n}%
;\tau,\omega \right)  .
\]
It follows that if $\mathfrak{K}_{1}\left(  x_{1},\ldots,x_{n};\tau \right)
>0,$ then $\limsup_{k\rightarrow \infty}\frac{\log \mathrm{Card}\left(
\mathrm{Rep}\left(  \mathcal{A},k\right)  /\backsimeq \right)  }{k^{2}} =
\infty$. If $\delta_{0}\left(  x_{1},\ldots,x_{n}; \tau \right)  $ denotes D.
Voiculescu's free entropy dimension, we know from \cite{V2} that
\[
\delta_{0}\left(  x_{1},\ldots,x_{n}; \tau \right)  \leq1+\mathfrak{K}%
_{1}\left(  x_{1},\ldots,x_{n};\tau \right) .
\]
This gives the following result, which shows that a matricially tracially
stable algebra may be forced to have a lot of representations of each large
finite dimension.

\begin{theorem}
\label{cardrep1} Suppose $\mathcal{A}=$C*$\left(  x_{1},\ldots,x_{n}\right)  $
is matricially tracially stable and $\tau$ is an embeddable tracial state on
$\mathcal{A}$ such that $1<\delta_{0}\left(  x_{1},\ldots,x_{n}\right)  .$
Then%
\[
\limsup_{k\rightarrow \infty}\frac{\log \mathrm{Card}\left(  \mathrm{Rep}\left(
\mathcal{A},k\right)  /\backsimeq \right)  }{k^{2}} = \infty.
\]

\end{theorem}

\medskip

Below using this theorem we give an example which shows that for not tracially nuclear $C^{*}%
$-algebras, the conditions 1) and 2) in Theorem \ref{conditions} are not
sufficient for being matricially tracially stable.

\begin{lemma}
\label{irreduciblepairsdense} \label{4}The set $\left \{  \left(  U,V\right)
\in \mathcal{U}_{n}\times \mathcal{U}_{n}:C^{\ast}\left(  U,V\right)
=\mathcal{M}_{n}\left(  \mathbb{C}\right)  \right \}  $ is norm dense in
$\mathcal{U}_{n}\times \mathcal{U}_{n}$.
\end{lemma}

\begin{proof}
Suppose $\left(  U,V\right)  \in \mathcal{U}_{n}\times \mathcal{U}_{n}$. Perturb
$U$ by an arbitrary small amount so that it is diagonal with no repeated
eigenvalues with respect to some orthonormal basis $\left \{  e_{1}%
,\ldots,e_{n}\right \}  $ and perturb $V$ by an arbitrary small amount so that
its eigenvalues are not repeated and one eigenvector has the form $\sum
_{k=1}^{n}\lambda_{k}e_{k}$ with $\lambda_{k}\neq0$ for $1\leq k\leq n$.
\end{proof}

\medskip

\noindent The following lemma states that every irreducible representation of
$(\pi_{1} \oplus \pi_{2}) \left(  \mathcal{A}\right)  $ must either factor
through $\pi_{1}$ or through $\pi_{2}$.

\begin{lemma}
\label{5}If $\pi=\pi_{1}\oplus \pi_{2}$ is a direct sum of unital
representations of a unital C*-algebra $\mathcal{A}$ and $\rho$ is an
irreducible representation of $\mathcal{A}$ such that $\ker \pi \subset \ker \rho
$, then either $\ker \pi_{1}\subset \ker \rho$ or $\ker \pi_{2}\subset \ker \rho$.
\end{lemma}
\begin{proof}
Suppose $\rho:\mathcal{A}\rightarrow B\left(  H\right)  $ is irreducible.
Assume, via contradiction, $a_{i}\in \ker \pi_{i}$ and $a_{i}\notin \ker \rho$ for
$i=1,2$. Then $a_{1}\mathcal{A}a_{2}\subseteq \ker \pi \subseteq \ker \rho$, so
$\rho \left(  a_{1}\right)  \rho \left(  \mathcal{A}\right)  \rho \left(
a_{2}\right)  =\left \{  0\right \}  ,$ and since $\rho \left(  \mathcal{A}%
\right)  ^{-wot}=B\left(  H\right)  ,$ we have $\rho \left(  a_{1}\right)
B\left(  H\right)  \rho \left(  a_{2}\right)  =\left \{  0\right \}  ,$ but
$\rho \left(  a_{1}\right)  \neq0\neq \rho \left(  a_{2}\right)  ,$ which is a contradiction.
\end{proof}

\begin{lemma}
\label{6}(Hoover \cite{Hoover}) If $\pi=\pi_{1}\oplus \pi_{2}$ is a direct sum
of unital representations of a unital C*-algebra $\mathcal{A}$, then
$\pi \left(  \mathcal{A}\right)  =\pi_{1}\left(  \mathcal{A}\right)  \oplus
\pi_{2}\left(  \mathcal{A}\right)  $ if and only if there is no irreducible
representation $\rho$ of $\mathcal{A}$ such that $\ker \pi_{1}\subset \ker \rho$
and $\ker \pi_{2}\subset \ker \rho$.
\end{lemma}
We will give a proof of it for the reader's convenience.
\begin{proof}
Let $\mathcal{J}_{1}=\left \{  b\in \pi_{1}\left(  \mathcal{A}%
\right)  :b\oplus0\in \pi \left(  \mathcal{A}\right)  \right \}  $ and
$\mathcal{J}_{2}=\left \{  b\in \pi_{2}\left(  \mathcal{A}\right)  :0\oplus
b\in \pi \left(  \mathcal{A}\right)  \right \}  $. Clearly,

$\pi \left(  \mathcal{A}\right)  =\pi_{1}\left(  \mathcal{A}\right)  \oplus
\pi_{2}\left(  A\right)  $ if and only if $1\in \mathcal{J}_{1}$ if and only if
$1\in \mathcal{J}_{2}$ if and only if $1\oplus1\in \mathcal{J}$ (where
$\mathcal{J=J}_{1}\oplus \mathcal{J}_{2}$). Hence, if $\pi \left(
\mathcal{A}\right)  \neq \pi_{1}\left(  \mathcal{A}\right)  \oplus \pi
_{2}\left(  A\right)  $, $\mathcal{J}$ is a proper ideal of $\pi \left(
\mathcal{A}\right)  $ and $\mathcal{J}_{k}$ is a proper ideal of $\pi
_{k}\left(  \mathcal{A}\right)  $ for $k=1,2$. Moreover, if $\pi \left(
a\right)  =\pi_{1}\left(  a\right)  \oplus \pi_{2}\left(  a\right)  $, we have:

$\pi \left(  a\right)  \in \mathcal{J}$ if and only if $\pi_{1}\left(  a\right)
\in \mathcal{J}_{1}$ if and only if $\pi_{2}\left(  a\right)  \in
\mathcal{J}_{2}$. Thus the algebras $\pi \left(  \mathcal{A}\right)
/\mathcal{J}$, $\pi_{1}\left(  \mathcal{A}\right)  /\mathcal{J}_{1}$ and
$\pi_{2}\left(  \mathcal{A}\right)  /\mathcal{J}_{2}$ are isomorphic to a
unital C*-algebra $\mathcal{D}$. Thus there are surjective homomorphisms
$\rho:\pi \left(  \mathcal{A}\right)  \rightarrow \mathcal{D}$ and $\rho_{k}%
:\pi_{k}\left(  \mathcal{A}\right)  \rightarrow \mathcal{D}$ such that
\[
\rho \circ \pi=\rho_{1}\circ \pi_{1}=\rho_{2}\circ \pi_{2}.
\]
If we choose an irreducible representation $\gamma$ of $\mathcal{D}$ and
replace $\rho,\rho_{1},\rho_{2}$ with $\gamma \circ \rho,\gamma \circ \rho
_{1},\gamma \circ \rho_{2}$ we get irreducible representations.
\end{proof}

\begin{theorem}
\label{NonNuclearExample} There exists a $C^{*}$-algebra which satisfies the
conditions 1) and 2) of Theorem \ref{conditions} but is not matricially
tracially stable.
\end{theorem}

\begin{proof}
Let $C_{r}^{\ast}\left(  \mathbb{F}_{2}\right)  $ be the reduced free group
C*-algebra. Then $C_{r}^{\ast}\left(  \mathbb{F}_{2}\right)  $ has a unique
tracial state $\tau$ and two canonical unitary generators $U,V$. D. Voiculescu
\cite{V1} proved that
\[
\delta_{0}\left(  U,V;\tau \right)  =2.
\]
It was also shown by U. Haagerup and S. Thorbj\o rnsen
\cite{HaagerupThorbjornsen} that  there are sequences $\left \{  U_{k}\right \}
,$ $\left \{  V_{k}\right \}  $ of unitary matrices with $U_{k},V_{k}%
\in \mathcal{M}_{k}\left(  \mathbb{C}\right)  $ for each $k\in \mathbb{N}$ such
that, for every $\ast$-polynomial $p\left(  x,y\right)  $%
\begin{equation}
\label{HaagerupThor1}\lim_{k\rightarrow \infty}\left \Vert p\left(  U_{k}%
,V_{k}\right)  \right \Vert =\left \Vert p\left(  U,V\right)  \right \Vert .
\end{equation}
It follows from the uniqueness of a trace on $C_{r}^{\ast}\left(  \mathbb{F}_{2}\right)  $ that, for every $\ast
$-polynomial $p\left(  s,t\right)  $,
\begin{equation}
\label{HaagerupThor2}\lim_{k\rightarrow \infty}\tau_{k}\left(  p\left(
U_{k},V_{k}\right)  \right)  =\tau \left(  p\left(  U,V\right)  \right) .
\end{equation}
By Lemma \ref{irreduciblepairsdense} we can assume that each pair $\left(
U_{k},V_{k}\right)  $ is irreducible, i.e., C*$\left(  U_{k},V_{k}\right)
=\mathcal{M}_{k}\left(  \mathbb{C}\right)  $. Let $U_{\infty}=U_{1}\oplus
U_{2}\oplus \cdots$ and $V_{\infty}=V_{1}\oplus V_{2}\oplus \cdots$ , and let
$\mathcal{A}=C^{\ast}\left(  U_{\infty},V_{\infty}\right)  $. Clearly
$\mathcal{A}$ is an RFD C*-algebra. Moreover, for each $k\in \mathbb{N}$,
$\mathcal{A}$ has, up to unitary equivalence, exactly one irreducible
representation $\pi_{k}$ of dimension $k$, namely the one with $\pi_{k}\left(
U_{\infty}\right)  =U_{k}$ and $\pi_{k}\left(  V_{\infty}\right)  =V_{k}$.
Since each $k$-dimensional representation of $\mathcal{A}$ is unitarily
equivalent to a direct sum of at most $k$ irreducible representations of
dimension at most $k,$ it follows that%
\[
\mathrm{Card}\left(  \mathrm{Rep}\left(  \mathcal{A},k\right)  /\backsimeq
\right)  \leq k^{k}.
\]
Thus
\[
\limsup_{k\rightarrow \infty}\frac{\log \mathrm{Card}\left(  \mathrm{Rep}\left(
\mathcal{A},k\right)  /\backsimeq \right)  }{k^{2}}\leq \limsup_{k\rightarrow
\infty}\frac{\log k}{k}=0.
\]
However, there is a unital $\ast$-homomorphism $\pi:\mathcal{A}\rightarrow
C_{r}^{\ast}\left(  \mathbb{F}_{2}\right)  $ such that $\pi \left(  U_{\infty
}\right)  =U$ and $\pi \left(  V_{\infty}\right)  =V$. Thus $\tau \circ \pi$ is
an embeddable tracial state on $\mathcal{A}$ and since%
\[
\delta_{0}\left(  U_{\infty},V_{\infty};\tau \circ \pi \right)  =\delta
_{0}\left(  U,V;\tau \right)  =2,
\]
it follows from Theorem \ref{cardrep1} that $\mathcal{A}$ is not matricially
tracially stable.

\textbf{Claim 1}: $\pi_{n}$ does not factor through $\sum_{0\leq k\neq
n<\infty}^{\oplus}\pi_{k}$.

Proof: Assume, via contradiction, the claim is false. Since $(U_{n}, V_{n})$
is an irreducible pair for $n\in \mathbb{N}$, it follows that $\pi_{n}$ does
not factor through $\sum_{0\leq k<N, k \neq n}^{\oplus}\pi_{k}$ for each
positive integer $N$. Hence, by Lemma \ref{5}, it follows that $\pi_{n}$
factors through $\sum_{N\leq k\neq n<\infty}^{\oplus}\pi_{k}$, for each
positive integer $N$. However, since by (\ref{HaagerupThor1}), for every
$a\in \mathcal{A}$, we have%
\[
\left \Vert \pi_{k}\left(  a\right)  \right \Vert \rightarrow \left \Vert
\pi \left(  a\right)  \right \Vert ,
\]
we see that, for every $a\in \mathcal{A}$,%
\[
\left \Vert \pi_{n}\left(  a\right)  \right \Vert \leq \left \Vert \pi \left(
a\right)  \right \Vert .
\]
This means that $\pi_{n}$ factors through $\pi,$ which contradicts the fact
that $C_{r}^{\ast}\left(  \mathbb{F}_{2}\right)  $ has no finite-dimensional
representations. Thus Claim 1 must be true.

\textbf{Claim 2}: $\mathcal{J}=\sum_{k=1}^{\infty}\mathcal{M}_{k}\left(
\mathbb{C}\right)  \subset \mathcal{A}. $

Proof: It is sufficient to show that the sequence $\left \{
\begin{array}
[c]{cc}%
Id_{k} & \text{if }k=n\\
0 & \text{otherwise}%
\end{array}
\right.  $ belongs to $\mathcal{A}$. Since $id_{\mathcal{A}} = \oplus \pi_{k}$,
the claim follows from Claim 1 and Lemma \ref{6}.

\medskip \noindent Clearly $\mathcal{A}/\mathcal{J}$ is isomorphic to
$C^{*}_{r}\left(  \mathbb{F}_{2}\right)  .$ Thus any factor representation
of $\mathcal{A}$ must either be finite-dimensional or be factorable through
the representation $\pi.$ Since the extreme tracial states are the factor
tracial states \cite{DonExtreme}, we see that the extreme tracial states on
$\mathcal{A}$ are%
\[
\left \{  \tau \circ \pi \right \}  \cup \left \{  \tau_{k}\circ \pi_{k}%
:k\in \mathbb{N}\right \}  .
\]
By (\ref{HaagerupThor2}) we see that both conditions in Theorem
\ref{conditions} are satisfied, but $\mathcal{A}$ is not matricially tracially stable.
\end{proof}

\section{$C^{*}$-tracial stability}

The following lemma must be very well known. We give a proof of it because of
lack of convenient references.

\begin{lemma}
\label{Lebesgue}Suppose $\left(  X,d\right)  $ is a compact metric space and
$\tau$ is a state on $C\left(  X\right)  $. Let $\sigma$ be the state on
$C\left[  0,1\right]  $ given by%
\[
\sigma \left(  f\right)  =\int_{0}^{1}f\left(  t\right)  dt.
\]
Then, for any non-trivial ultrafilter $\alpha$ on $\mathbb{N}$ there is a
unital $\ast$-homomorphism $\pi:C\left(  X\right)  \rightarrow%
%TCIMACRO{\dprod _{n\in\mathbb{N}}^{\alpha}}%
%BeginExpansion
{\displaystyle \prod_{n\in \mathbb{N}}^{\alpha}}
%EndExpansion
\left(  C\left[  0,1\right]  ,\sigma \right) ,$ such that $\sigma_{\alpha}%
\circ \pi= \tau.$
\end{lemma}
\begin{proof}
We know from \cite{Hadwin-Li} that $%
%TCIMACRO{\dprod _{n\in\mathbb{N}}^{\alpha}}%
%BeginExpansion
{\displaystyle \prod_{n\in \mathbb{N}}^{\alpha}}
%EndExpansion
\left(  C\left[  0,1\right]  , \sigma \right)  $ is a von Neumann algebra and
that $%
%TCIMACRO{\dprod _{n\in\mathbb{N}}^{\alpha}}%
%BeginExpansion
{\displaystyle \prod_{n\in \mathbb{N}}^{\alpha}}
%EndExpansion
\left(  C\left[  0,1\right]  ,\sigma \right)  =%
%TCIMACRO{\dprod _{n\in\mathbb{N}}^{\alpha}}%
%BeginExpansion
{\displaystyle \prod_{n\in \mathbb{N}}^{\alpha}}
%EndExpansion
\left(  L^{\infty}\left[  0,1\right]  ,\sigma \right)  $. We know that there is
a probability Borel measure $\mu$ on $X$ such that, for every $f\in C\left(
X\right)  $,%
\[
\tau \left(  f\right)  =\int_{X}fd\mu \text{.}%
\]
Choose a dense subset $\left \{  f_{1},f_{2},\ldots \right \}  $ of $C\left(
X\right)  $. For each $n\in \mathbb{N}$ we can find a Borel partition $\left \{
E_{n,1},\ldots,E_{n,k_{n}}\right \}  $ of $X$ so that each $E_{n,j}$ has
sufficiently small diameter and points $x_{n,j}\in E_{n,j}$ for $1\leq j\leq
k_{n}$ such that, for $1\leq m\leq n$ and for every $x\in X$,%
\[
\left \vert f_{m}\left(  x\right)  -\left(  \sum_{j=1}^{k_{n}}f_{m}\left(
x_{n,j}\right)  \chi_{E_{n,j}}\right)  \left(  x\right)  \right \vert \leq
\frac{1}{n}.
\]
For each $n\in \mathbb{N}$ we can partition $\left[  0,1\right]  $ into
intervals $\left \{  I_{n,j}:1\leq j\leq k_{n}\right \}  $ so that%
\[
\sigma \left(  \chi_{I_{n,j}}\right)  =\mu \left(  E_{n,j}\right)  .
\]
For each $n\in \mathbb{N}$ we define a unital $\ast$-homomorphism $\pi
_{n}:C\left(  X\right)  \rightarrow L^{\infty}\left[  0,1\right]  $ by%
\[
\pi_{n}\left(  f\right)  =\sum_{j=1}^{k_{n}}f\left(  x_{n,j}\right)
\chi_{I_{n,j}}.
\]
We define $\pi:C\left(  X\right)  \rightarrow%
%TCIMACRO{\dprod _{n\in\mathbb{N}}^{\alpha}}%
%BeginExpansion
{\displaystyle \prod_{n\in \mathbb{N}}^{\alpha}}
%EndExpansion
\left(  L^{\infty}\left[  0,1\right]  ,\sigma \right)  $ by%
\[
\pi \left(  f_{m}\right)  =\left \{  \pi_{n}\left(  f_{m}\right)  \right \}
_{\alpha}%
\]
It is clear, for $m\in \mathbb{N}$, that%
\[
\sigma_{\alpha}\left(  \pi \left(  f_{m}\right)  \right)  =\lim_{n\rightarrow
\alpha}\sigma \left(  \pi_{n}\left(  f_{m}\right)  \right)  =\lim
_{n\rightarrow \alpha}\sum_{j=1}^{k_{n}}f_{m}\left(  x_{n,j}\right)
\mu({E_{n,j})}=\tau \left(  f_{m}\right)  \text{.}%
\]
Since $\left \{  f_{1},f_{2},\ldots \right \}  $ is dense in $C\left(  X\right)
$, we see that $\tau=\sigma_{\alpha}\circ \pi$.
\end{proof}

\bigskip

We say that a topological space $X$ is \emph{approximately path-connected} if,
for each collection $\left \{  V_{1},\ldots,V_{s}\right \}  $ of nonempty open
subsets of $X$ there is a continuous function $\gamma:\left[  0,1\right]
\rightarrow X$ such that%
\[
\gamma \left[  \left(  0,1\right)  \right]  \cap V_{k}\neq \varnothing
\]
for $1\leq k\leq s.$

Equivalently, traversing the above path $\gamma$ back and forth a finite
number of times, given $0<a_{1}<b_{1}<\cdots<a_{s}<b_{s}<1$ we can find a path
$\gamma$ such that%
\[
\gamma \left(  \left(  a_{k},b_{k}\right)  \right)  \subset V_{k}%
\]
for $1\leq k\leq s$. Indeed, we first find a path $\gamma_{1}$ and
$0<c_{1}<d_{1}<\cdots<c_{s}<d_{s}<1$ such that
\[
\gamma_{1}\left(  \left(  c_{k},d_{k}\right)  \right)  \subset V_{k}%
\]
for $1\leq k\leq s$, and then compose $\gamma_{1}$ with a homeomorphism on
$\left[  0,1\right]  $ sending $a_{k}$ to $c_{k}$ and $b_{k}$ to $d_{k}$ for
$1\leq k\leq s.$

To prove that $X$ is approximately path-connected when $X$ is Hausdorff, it
clearly suffices to restrict to the case in which $\left \{  V_{1},\ldots
,V_{s}\right \}  $ is disjoint (Choose $x_{k}\in V_{k}$ for each $k$ and chose
an open set $W_{k}\subseteq V_{k}$ with $x_{k}\in W_{k}$ such that
$x_{k}=x_{j}\Rightarrow W_{j}=W_{k}$ and such that $\left \{  W_{k}:1\leq k\leq
s\right \}  $ (without repetitions) is disjoint, then consider $\left \{
W_{k}:1\leq k\leq s\right \}  $.)

The following facts are elementary:

\begin{enumerate}
\item Every approximately path-connected space is connected

\item A continuous image of an approximately path-connected space is
approximately path-connected

\item A cartesian product of approximately path-connected spaces is
approximately path-connected.
\end{enumerate}

An interesting example of a compact approximately path connected metric space
in $\mathbb{R}^{2}$ is
\[
A=\left \{  \left(  x,\sin \left(  \frac{1}{x}\right)  \right)  :0<x\leq
1\right \}  \cup \left \{  \left(  0,y\right)  :1-\leq y\leq1\right \}  .
\]
Note that $A\cup B$ with $B=\left \{  \left(  x,0\right)  :-1\leq
x\leq0\right \}  $ is not approximately path-connected. In particular, $A$ and
$B$ are approximately path-connected and $A\cap B\neq \varnothing$, but $A\cup
B$ is not approximately path-connected.

For compact Hausdorff spaces we have a characterization of approximate
path-connectedness which later will be used to characterize $C^{*}$-tracial
stability for commutative $C^{*}$-algebras.

\begin{theorem}
\label{ApprPathConnectedEquivalentConditions} Suppose $X$ is a compact
Hausdorff space. The following are equivalent:

\begin{enumerate}
\item $X$ is approximately path-connected.

\item If $\mathcal{A}$ is a unital C*-algebra, $\mathcal{B}$ is an
AW*-algebra, $\eta:\mathcal{A}\rightarrow \mathcal{B}$ is a surjective unital
$\ast$-homomorphism, and $\pi:C\left(  X\right)  \rightarrow \mathcal{B}$ is a
unital $\ast$-homomorphism, then there is is a net $\left \{  \pi_{\lambda
}\right \}  $ of unital $\ast$-homomorphisms from $C\left(  X\right)  $ to
$\mathcal{A}$ such that, for every $f\in C\left(  X\right)  ,$%
\[
\left(  \eta \circ \pi_{\lambda}\right)  \left(  f\right)  \rightarrow \pi \left(
f\right)
\]
in the weak topology on $\mathcal{B}$.

\item If $\mathcal{A}$ is a unital C*-algebra, $\mathcal{B}$ is an W*-algebra,
$\eta:\mathcal{A}\rightarrow \mathcal{B}$ is a surjective unital $\ast
$-homomorphism, and $\pi:C\left(  X\right)  \rightarrow \mathcal{B}$ is a
unital $\ast$-homomorphism, then there is is a net $\left \{  \pi_{\lambda
}\right \}  $ of unital $\ast$-homomorphisms from $C\left(  X\right)  $ to
$\mathcal{A}$ such that, for every $f\in C\left(  X\right)  ,$%
\[
\left(  \eta \circ \pi_{\lambda}\right)  \left(  f\right)  \rightarrow \pi \left(
f\right)
\]
in the ultra*-strong topology on $\mathcal{B}$.

\item For every regular Borel probability measure $\nu$ on $X,$ there is a net
$\left \{  \rho_{\omega}\right \}  $ of unital $\ast$-homomorphisms from
$C\left(  X\right)  $ into $C\left[  0,1\right]  $ such that, for every $f\in
C\left(  X\right)  ,$%
\[
\int_{X}fd\nu=\lim_{\omega}\int_{0}^{1}\rho_{\omega}\left(  f\right)  \left(
x\right)  dx.
\]

\end{enumerate}
\end{theorem}

\begin{proof}
$\left(  1\right)  \Rightarrow \left(  2\right)  $. Suppose $X$ is
approximately path-connected and suppose $\mathcal{A},\mathcal{B},\eta,$ and
$\pi$ are is in $\left(  2\right)  $. Let $\Lambda$ be the collection of all
tuples of the form $\lambda=\left(  S_{\lambda},E_{\lambda},\varepsilon
_{\lambda}\right)  $ where $S_{\lambda}$ is a finite set of states on
$\mathcal{B}$, $E_{\lambda}$ is a finite subset of $C\left(  X\right)  $ of
functions from $X$ to $\left[  0,1\right]  $, and $\varepsilon_{\lambda}>0$.
Suppose $\lambda \in \Lambda$. Clearly, $\pi \left(  C\left(  X\right)  \right)
$ is an abelian selfadjoint C*-subalgebra of $\mathcal{B}$, and therefore
there is a maximal abelian C*-subalgebra $\mathcal{M}$ of $\mathcal{B}$ such
that $\pi \left(  C\left(  X\right)  \right)  \subseteq \mathcal{M}$. Since
$\mathcal{B}$ is an AW*-algebra, there is a commuting family $\mathcal{P}$ of
projections in $\mathcal{B}$ such that $\mathcal{M}=C^{\ast}\left(
\mathcal{P}\right)  $. Since $C^{\ast}\left(  E_{\lambda}\right)  $ is a
separable unital C*-subalgebra of $C\left(  X\right)  $, the maximal ideal
space of $C^{\ast}\left(  E_{\lambda}\right)  $ is a compact metric space
$X_{\lambda}$ and there is a surjective continuous function $\zeta_{\lambda
}:X\rightarrow X_{\lambda}$. It follows that $X_{\lambda}$ is approximately
path-connected. Note that if we identify $C^{\ast}\left(  E_{\lambda}\right)
$ with $C\left(  X_{\lambda}\right)  $ we can view a function $f\in C^{\ast
}\left(  E_{\lambda}\right)  $ as both a function on $X$ and on $X_{\lambda}$
and we have $f=f\circ \zeta_{\lambda}$. Now $\pi \left(  C^{\ast}\left(
E_{\lambda}\right)  \right)  =\pi \left(  C\left(  X_{\lambda}\right)  \right)
$ is a separable subalgebra of $C^{\ast}\left(  \mathcal{P}\right)  $ so there
is a countable subset $\mathcal{P}_{\lambda}\subseteq \mathcal{P}$ such that
$\pi \left(  C\left(  X_{\lambda}\right)  \right)  \subseteq C^{\ast}\left(
\mathcal{P}_{\lambda}\right)  $. If follows from von Neumann's argument that
there is a $w_{\lambda}=w_{\lambda}^{\ast}$ with $\sigma \left(  w_{\lambda
}\right)  $ is a totally disconnected subset of $\left(  0,1\right)  $ such
that $C^{\ast}\left(  w_{\lambda}\right)  =C^{\ast}\left(  \mathcal{P}%
_{\lambda}\right)  $. Let $\psi_{\lambda}=\sum_{\psi \in S_{\lambda}}\psi$ and
let $n_{\lambda}= \sharp(S_{\lambda})$. Then there is a measure $\mu_{\lambda
}$ on $\sigma \left(  w_{\lambda}\right)  $ such that, for every $h\in C\left(
\sigma \left(  w_{\lambda}\right)  \right)  ,$ we have%
\[
\int_{\sigma \left(  w_{\lambda}\right)  }hd\mu_{\lambda}=\psi_{\lambda}\left(
h\left(  w_{\lambda}\right)  \right)  .
\]
For each $f\in E_{\lambda}$ there is a function $h_{f}\in C\left(
\sigma \left(  w_{\lambda}\right)  \right)  $ such that $\pi \left(  f\right)
=h_{f}\left(  w_{\lambda}\right)  $ and there is a function $\hat{h}_{f}\in
C\left[  0,1\right]  $ with $0\leq \hat{h}_{f}\leq1$ such that $\hat{h}%
_{f}|_{\sigma \left(  w_{\lambda}\right)  }=h_{f}$. We can view $\mu_{\lambda}$
as a Borel measure on $\left[  0,1\right]  $ by defining $\mu_{\lambda}\left(
\left[  0,1\right]  \backslash \sigma \left(  w_{\lambda}\right)  \right)  =0$.
Clearly, $\mu_{\lambda}\left(  \left[  0,1\right]  \right)  =\psi_{\lambda
}\left(  1\right)  =n_{\lambda}$. Since $\left \{  \hat{h}_{f}:f\in E_{\lambda
}\right \}  $ is equicontinuous and $\left[  0,1\right]  \backslash
\sigma \left(  w_{\lambda}\right)  $ is dense in $\left[  0,1\right]  $, we can
find $0<a_{1}<b_{1}<a_{2}<b_{2}<\cdots<a_{s}<b_{s}<1$ and $r_{1},\ldots
,r_{s}\in \sigma \left(  w_{\lambda}\right)  $ such that, for every $f\in E$ and
$1\leq j\leq s$,%
\[
\left \vert \hat{h}_{f}\left(  t\right)  -\hat{h}_{f}\left(  r_{j}\right)
\right \vert <\varepsilon_{\lambda}/4n_{\lambda}\text{ when }t\in \left(
a_{j},b_{j}\right)
\]
and such that
\[
\mu_{\lambda}\left(  \left[  0,1\right]  \backslash \cup_{j=1}^{s}\left(
a_{j},b_{j}\right)  \right)  <\varepsilon_{\lambda}/4.
\]
We know $\hat{h}_{f}\left(  r_{j}\right)  =h_{f}\left(  r_{j}\right)  $ for
all $f\in E_{\lambda}$ and $1\leq j\leq s$. Since $\pi:C\left(  X_{\lambda
}\right)  \rightarrow C\left(  \sigma \left(  w_{\lambda}\right)  \right)  $,
there is a $y_{j}\in X_{\lambda}$ for $1\leq j\leq s$ such that, for every
$f\in E_{\lambda}$, $f\left(  y_{j}\right)  =h_{f}\left(  r_{j}\right)  .$ On
the other hand, we see that there is an $x_{j}\in X$ for $1\leq j\leq s$ such
that $\zeta_{\lambda}\left(  x_{j}\right)  =y_{j}$. Hence, we have
\[
\left \vert \hat{h}_{f}\left(  t\right)  -f\left(  x_{j}\right)  \right \vert
<\varepsilon_{\lambda}/4n_{\lambda}\text{ when }t\in \left(  a_{j}%
,b_{j}\right)  .
\]
We next choose an open set $V_{j}\subseteq X$ with $x_{j}\in V_{j}$ such that,
for every $f\in E_{\lambda}$ and every $x\in V,$ we have%
\[
\left \vert f\left(  x\right)  -f\left(  x_{j}\right)  \right \vert
<\varepsilon_{\lambda}/4n_{\lambda}.
\]
We now use the fact that $X$ is approximately path-connected to find a
continuous function $\gamma_{\lambda}:\left[  0,1\right]  \rightarrow X$ such
that, for $1\leq j\leq s$,%
\[
\gamma_{\lambda}\left(  \left(  a_{j},b_{j}\right)  \right)  \subseteq
V_{j}\text{.}%
\]
We have, for each $f\in E_{\lambda}$, each $1\leq j\leq s$, and each
$t\in \left(  a_{j},b_{j}\right)  $%
\[
\left \vert \hat{h}_{f}\left(  t\right)  -f\circ \gamma_{\lambda}\left(
t\right)  \right \vert <\varepsilon_{\lambda}/2n_{\lambda}%
\]
and for $t\in \left[  0,1\right]  \backslash \cup_{j=1}^{s}\left(  a_{j}%
,b_{j}\right)  ,$%
\[
\left \vert \hat{h}_{f}\left(  t\right)  -f\circ \gamma_{\lambda}\left(
t\right)  \right \vert \leq2.
\]
Hence, for every $f\in E_{\lambda}$%
\begin{multline*}
\left \vert \psi_{\lambda}\left(  \pi \left(  f\right)  \right)  -\psi_{\lambda
}\left(  \left(  f\circ \gamma_{\lambda}\right)  \left(  w_{\lambda}\right)
\right)  \right \vert \leq \int \left \vert \hat{h}_{f}-\left(  f\circ
\gamma_{\lambda}\right)  \left(  w_{\lambda}\right)  \right \vert d\mu
_{\lambda}<\\
\left(  \varepsilon_{\lambda}/2n_{\lambda}\right)  \mu_{\lambda}\left(
\left[  0,1\right]  \right)  +2\mu_{\lambda}\left(  \left[  0,1\right]
\backslash \cup_{j=1}^{s}\left(  a_{j},b_{j}\right)  \right)  <\varepsilon
_{\lambda}.
\end{multline*}
Since $\psi_{\lambda}=\sum_{\psi \in S_{\lambda}}\psi$, the measure
$\mu_{\lambda}$ is a sum of probability measures, one corresponding to each
$\psi \in S_{\lambda}$. We therefore have, for every $f\in E_{\lambda}$ and
every $\psi \in S_{\lambda}$,%
\[
\left \vert \psi \left(  \pi \left(  f\right)  \right)  -\psi \left(  \left(
f\circ \gamma_{\lambda}\right)  \left(  w_{\lambda}\right)  \right)
\right \vert \leq \int \left \vert \hat{h}_{f}-\left(  f\circ \gamma_{\lambda
}\right)  \left(  w_{\lambda}\right)  \right \vert d\mu_{\lambda}%
<\varepsilon_{\lambda}.
\]
We can choose $v_{\lambda}\in \mathcal{A}$ with $0\leq v_{\lambda}\leq1$ such
that $\eta \left(  v_{\lambda}\right)  =w_{\lambda}$. We define a unital $\ast
$-homomorphism $\pi_{\lambda}:C\left(  X\right)  \rightarrow \mathcal{A}$ by
$\pi_{\lambda}\left(  f\right)  =\left(  f\circ \gamma_{\lambda}\right)
\left(  v_{\lambda}\right)  $. Hence, for every $f\in C\left(  X\right)  $,
\[
\left(  \eta \circ \pi_{\lambda}\right)  \left(  f\right)  =\eta \left(  \left(
f\circ \gamma_{\lambda}\right)  \left(  v_{\lambda}\right)  \right)  =\left(
f\circ \gamma_{\lambda}\right)  \left(  w_{\lambda}\right)  .
\]
Hence, for every $f\in E_{\lambda}$ and every $\psi \in S_{\lambda}$ we have%
\[
\left \vert \psi \left(  \left(  \eta \circ \pi_{\lambda}\right)  \left(
f\right)  \right)  -\psi \left(  \pi \left(  f\right)  \right)  \right \vert
<\varepsilon_{\lambda}.
\]
It follows, for every $f\in C\left(  X\right)  $ with $0\leq f\leq1$ and every
state $\psi$ on $\mathcal{B}$, that%
\[
\lim_{\lambda}\psi \left(  \left(  \eta \circ \pi_{\lambda}\right)  \left(
f\right)  \right)  =\psi \left(  \pi \left(  f\right)  \right)  .
\]
Since every $g\in C\left(  X\right)  $ is a linear combination of $f$'s with
$0\leq f\leq1$ and every continuous linear functional on $\mathcal{B}$ is a
linear combination of states, we see, for every $f\in C\left(  X\right)  $,
that $\left(  \eta \circ \pi_{\lambda}\right)  \left(  f\right)  \rightarrow
\pi \left(  f\right)  $ in the weak topology on $\mathcal{B}$.

$\left(  2\right)  \Rightarrow \left(  3\right)  .$ Since every W*-algebra is
an AW*-algebra, it is clear that we can find a net $\left \{  \pi_{\lambda
}\right \}  $ as in $\left(  2\right)  $. Thus, for every $f\in C\left(
X\right)  $, $\left(  \eta \circ \pi_{\lambda}\right)  \left(  f\right)
\rightarrow \pi \left(  f\right)  $ and
\[
\left(  \eta \circ \pi_{\lambda}\right)  \left(  f\right)  ^{\ast}\left(
\eta \circ \pi_{\lambda}\right)  \left(  f\right)  =\left(  \eta \circ
\pi_{\lambda}\right)  \left(  f^{\ast}f\right)  \rightarrow \pi \left(
f\right)  ^{\ast}\pi \left(  f\right)
\]
ultraweakly. Hence we have $\left(  \eta \circ \pi_{\lambda}\right)  \left(
f\right)  \rightarrow \pi \left(  f\right)  $ in the ultra*-strong topology on
$\mathcal{B}$.

$\left(  3\right)  \Rightarrow \left(  4\right)  $. By Lemma \ref{Lebesgue}
there exists a unital $\ast$-homomorphism $\pi: C(X) \to%
%TCIMACRO{\dprod _{n\in\mathbb{N}}^{\alpha}}%
%BeginExpansion
{\displaystyle \prod_{n\in \mathbb{N}}^{\alpha}}
%EndExpansion
\left(  C\left[  0,1\right]  ,\sigma \right)  $ such that
\begin{equation}
\label{3)to4)}\sigma_{\alpha}\circ \pi(f) = \int_{X} f d\nu,
\end{equation}
for each $f\in C(X)$. Here a state $\sigma$ on $C[0, 1]$ is given by
$\sigma(g) = \int_{0}^{1} g dx.$ Let $\eta: \prod C[0, 1] \to%
%TCIMACRO{\dprod _{n\in\mathbb{N}}^{\alpha}}%
%BeginExpansion
{\displaystyle \prod_{n\in \mathbb{N}}^{\alpha}}
%EndExpansion
\left(  C\left[  0,1\right]  ,\sigma \right)  $ be the canonical surjection. By
3) there is a net $\left \{  \pi_{\lambda}\right \}  $ of unital $\ast
$-homomorphisms from $C\left(  X\right)  $ to $\prod C[0, 1]$ such that, for
every $f\in C\left(  X\right)  ,$%
\[
\left(  \eta \circ \pi_{\lambda}\right)  \left(  f\right)  \rightarrow \pi \left(
f\right)
\]

\noindent ultra*-strongly. By Lemma \ref{LimitOfLiftable} there exist unital
$\ast$-homomorphisms $\rho_{n}: C(X) \to C[0, 1]$ such that
\[
\pi(f) = \eta((\rho_{n}(f))_{n\in \mathbb{N}}),
\]
for each $f\in C(X)$. By (\ref{3)to4)})
\[
\int_{X} f d\nu= \lim_{\alpha} \int_{0}^{1} \rho_{n}(f) dx.
\]
The sequence $\int_{0}^{1} \rho_{n}(f) dx$ contains a subnet $\int_{0}^{1}
\rho_{n_{\omega}}(f) dx$ which is an ultranet. This ultranet has to converge
to $\int_{X} f d\nu$. Indeed, if $\lim_{\alpha} t_{n} =t$ and $t_{n_{\omega}}$
is an ultranet, then $\lim_{\omega} t_{n_{\omega}} =t$. (Proof: for any
$\epsilon> 0$ the set $\{n_{\omega}\;|\; |t_{n_{\omega}} - t|< \epsilon \}$ is
infinite, otherwise $\{n_{\omega}\;|\;|t_{n_{\omega}} - t|< \epsilon
\} \notin \alpha$ and we would have $\emptyset= \{n_{\omega}\;|\; |t_{n_{\omega
}} - t|\ge \epsilon \} \bigcap \{n\;| \; |t_{n} - t|< \epsilon \} \in \alpha$.
Hence $t$ is an accumulation point for $\{t_{n_{\omega}}\}$ and since
$\{t_{n_{\omega}}\}$ is an ultranet, $t$ is its limit.) Thus we have
\[
\int_{X} f d\nu= \lim_{\omega} \int_{0}^{1} \rho_{n_{\omega}}(f) dx.
\]
\bigskip

$\left(  4\right)  \Rightarrow \left(  1\right)  $. Assume $\left(  4\right)  $
is true. Suppose $V_{1},V_{2},\ldots,V_{s}$ are nonempty open subsets of $X$.
There is no harm in assuming $\left \{  V_{1},\ldots,V_{s}\right \}  $ is
disjoint. For each $1\leq j\leq s$ we can choose $x_{k}\in V_{k}$ and a
continuous function $h_{j}:X\rightarrow \left[  0,1\right]  $ such that
$h_{j}\left(  x_{j}\right)  =1$ and $h_{j}|_{X\backslash V_{j}}=0$. Let
$\mu=\frac{1}{s}\sum_{j=1}^{s}\delta_{x_{j}}$. Then $\mu$ is a probability
measure with $\int_{X}h_{j}d\mu=\frac{1}{s}$. It follows from $\left(
4\right)  $ that there is a unital $\ast$-homomorphism $\rho:C\left(
X\right)  \rightarrow C\left[  0,1\right]  $ such that%
\[
\int_{0}^{1}\rho \left(  h_{j}\right)  \left(  x\right)  dx\neq0
\]
for $1\leq k\leq s$. However, there must be a continuous map $\gamma:\left[
0,1\right]  \rightarrow X$ such that $\pi \left(  f\right)  =f\circ \gamma$ for
every $f\in C\left(  X\right)  .$ For each $1\leq j\leq s,$ $0\neq \int_{0}%
^{1}\left(  h_{j}\circ \gamma \right)  \left(  x\right)  dx$ implies that there
is a $t_{j}\in \left[  0,1\right]  $ such that $h_{j}\left(  \gamma \left(
t_{j}\right)  \right)  \neq0.$ Thus, by the definition of $h_{j}$, we have
$\gamma \left(  t_{j}\right)  \in V_{j}$ for $1\leq j\leq s.$ Therefore, $X$ is
approximately path-connected.
\end{proof}

\bigskip

\noindent {\bf Remark.}
\textrm{In statement $\left(  2\right)  $ in Theorem
\ref{ApprPathConnectedEquivalentConditions}, if we view $\mathcal{B}\subseteq
B\left(  H\right)  $ as the universal representation (that is the direct sum
of all irreducible representations), then the weak operator topology on
$\mathcal{B}$ is the weak (and the ultraweak) topology on $\mathcal{B}$. Thus,
if $\left(  \eta \circ \pi_{\lambda}\right)  \left(  f\right)  \rightarrow
\pi \left(  f\right)  $ and
\[
\left[  \left(  \eta \circ \pi_{\lambda}\right)  \left(  f\right)  \right]
^{\ast}\left[  \left(  \eta \circ \pi_{\lambda}\right)  \left(  f\right)
\right]  =\left(  \eta \circ \pi_{\lambda}\right)  \left(  f^{\ast}f\right)
\rightarrow \pi \left(  f^{\ast}f\right)  =\pi \left(  f\right)  ^{\ast}%
\pi \left(  f\right)
\]
weakly for every $f\in C\left(  X\right)  $ implies, for each $f\in C\left(
X\right)  $ that%
\[
\left(  \eta \circ \pi_{\lambda}\right)  \left(  f\right)  \rightarrow \pi \left(
f\right)
\]
is the $\ast$-strong operator topology in $B\left(  H\right)  .$}

\bigskip

%Since $C\left(  X\right)  $ is separable if and only if the compact Hausdorff
%space $X$ is metrizable we can replace nets with sequences and easily obtain
%the following theorem using the fact (\cite{Hadwin-Li}) that tracial ultraproducts
%are von Neumann algebras.
As a corollary we obtain a characterization of when a separable commutative
$C^{*}$-algebra is $C^{*}$-tracially stable.

\begin{theorem}
\label{C*-tracial} Suppose $\left(  X,d\right)  $\ is a compact metric space.
The following are equivalent:

\begin{enumerate}
\item $C\left(  X\right)  $ is $C^{*}$-tracially stable.

\item $X$ is approximately path-connected

\item For every state $\tau$ on $C\left(  X\right)  $ there is a sequence
$\pi_{n}:C\left(  X\right)  \rightarrow C\left[  0,1\right]  $ such that, for
every $f\in C\left(  X\right)  ,$%
\[
\tau \left(  f\right)  =\lim_{n\rightarrow \infty}\int_{0}^{1}\pi_{n}\left(
f\right)  \left(  x\right)  dx.
\]

\end{enumerate}
\end{theorem}

\begin{proof}
2) $\Leftrightarrow$ 3) by the equivalence of statements 1) and 4) in Theorem
\ref{ApprPathConnectedEquivalentConditions} and separability of $C(X)$.

3) $\Rightarrow$ 1): Let $\phi: C(X) \to \prod_{i\in I}^{\alpha} \left(
\mathcal{A}_{i},\rho_{i}\right)  $ be a unital $\ast$-homomorphism. By
\cite{Hadwin-Li} $\prod_{i\in I}^{\alpha} \left(  \mathcal{A}_{i},\rho
_{i}\right)  $ is a von Neumann algebra and by the equivalence 3)
$\Leftrightarrow$ 4) in Theorem \ref{ApprPathConnectedEquivalentConditions}
$\phi$ is a $\ast$-strong pointwise limit of liftable $\ast$-homomorphisms
from $C(X) \to \prod_{i\in I}^{\alpha} \left(  \mathcal{A}_{i},\rho_{i}\right)
$. By Lemma \ref{LimitOfLiftable} $\phi$ is liftable.

1) $\Rightarrow$ 3): By Lemma \ref{Lebesgue} there is a unital $\ast
$-homomorphism $\pi: C\left(  X\right)  \rightarrow%
%TCIMACRO{\dprod _{n\in\mathbb{N}}^{\alpha}}%
%BeginExpansion
{\displaystyle \prod_{n\in \mathbb{N}}^{\alpha}}
%EndExpansion
\left(  C\left[  0,1\right]  ,\sigma \right)  $ such that $\sigma_{\alpha}%
\circ \pi= \tau.$ By 1) we can lift it and obtain a sequence $\pi_{n}:C\left(
X\right)  \rightarrow C\left[  0,1\right]  $ such that, for every $f\in
C\left(  X\right)  ,$%
\[
\tau \left(  f\right)  =\lim_{\alpha}\int_{0}^{1}\pi_{n}\left(  f\right)
\left(  x\right)  dx.
\]

\noindent Taking a subnet, which is an ultranet we obtain (by the same
arguments as in the proof of the implication $3) \Rightarrow4)$ in Theorem
\ref{ApprPathConnectedEquivalentConditions}) that for any $f\in C(X)$
\[
\tau \left(  f\right)  =\lim_{\omega}\int_{0}^{1}\pi_{n_{\omega}}\left(
f\right)  \left(  x\right)  dx.
\]

\noindent Since $C(X)$ is separable, we can pass to a subsequence.
\end{proof}

\begin{corollary}
Suppose $\mathcal{A}$ is a unital commutative $C^{*}$-tracially stable
C*-algebra and $\mathcal{A}_{0}$ is a unital C*-subalgebra of $\mathcal{A}$.
Then $\mathcal{A}_{0}$ is $C^{*}$-tracially stable.
\end{corollary}

\begin{proof}
We can assume $\mathcal{A}=C\left(  X\right)  $ with $X$ an approximately
path-connected compact metric space. Also we can write $\mathcal{A}%
_{0}=C\left(  Y\right)  $, and, since $\mathcal{A}_{0}$ embeds into
$\mathcal{A}$, there is a continuous surjective map $\varphi:X\rightarrow Y.$
Thus $Y$ is approximately path-connected, which implies $\mathcal{A}_{0}$ is
C*-tracially stable.
\end{proof}

\bigskip

At this point there is little else we can say about C*-tracially stable
algebras, except that they do not have projections when there is a faithful
embeddable tracial state.

\begin{theorem}
Suppose $\mathcal{A}$ is a separable C*-tracially stable unital C*-algebra.
Then $\mathcal{A}/\mathcal{J}_{et}$ has no nontrivial projections.
\end{theorem}

\begin{proof}
Without loss of generality we can assume $\mathcal{J}_{et}=\left \{  0\right \}
$. Then $\mathcal{A}$ has a faithful embeddable tracial state $\sigma$. Then
there is a tracial embedding $\pi$ of $\left(  \mathcal{A},\sigma \right)  $
into $%
%TCIMACRO{\dprod _{n\in\mathbb{N}}^{\alpha}}%
%BeginExpansion
{\displaystyle \prod_{n\in \mathbb{N}}^{\alpha}}
%EndExpansion
\left(  C_{r}^{\ast}\left(  \mathbb{F}_{2}\right)  ,\tau \right)  $ for some
non-trivial ultrafilter $\alpha$ on $\mathbb{N}$. Indeed, $C_{r}^{\ast}\left(
\mathbb{F}_{2}\right)  $ has a unique trace $\tau$ and it is a subalgebra of
the factor von Neumann algebra $\mathcal{L}_{{\mathbb{F}}_{2}}$ so that
$\mathcal{L}_{\mathbb{F}_{2}}=W^{\ast}(C_{r}^{\ast}\left(  \mathbb{F}%
_{2}\right)  )$. It follows from the Kaplansky density theorem that the $||\;
\Vert_{2}$-closure of the unit ball of $C^{\ast}(\mathbb{F}_{2})$ is the unit
ball of $\mathcal{L}_{\mathbb{F}_{2}}$ which implies that
\[
\prod^{\alpha}(C^{\ast}(\mathbb{F}_{2}),\tau)=\prod^{\alpha}(\mathcal{L}%
_{\mathbb{F}_{2}},\tau).
\]
Since $\mathcal{L}_{\mathbb{F}_{2}}$ contains $M_{n}(\mathbb{C})$ for each
$n\in \mathbb{N}$, $\prod^{\alpha}(M_{n}(\mathbb{C}),\tau_{n})$ embeds into
$\prod^{\alpha}(C^{\ast}(\mathbb{F}_{2}),\tau)=\prod^{\alpha}(\mathcal{L}%
_{\mathbb{F}_{2}},\tau).$

Now since $\mathcal{A}$ is tracially stable, there is a sequence $\left \{
\pi_{n}\right \}  $ of unital $\ast$-homomorphisms from $\mathcal{A}$ into
$C_{r}^{\ast}\left(  \mathbb{F}_{2}\right)  $ such that, for every
$a\in \mathcal{A}$ ,%
\[
\sigma \left(  a\right)  =\lim_{n\rightarrow \alpha}\tau \left(  \pi_{n}\left(
a\right)  \right)  .
\]
Suppose $P$ is a projection in $\mathcal{A}$. Since $C_{r}^{\ast}\left(
\mathbb{F}_{2}\right)  $ contains no non-trivial projections, for each
$n\in \mathbb{N}$ $\pi_{n}\left(  P\right)  =0$ or $\pi_{n}\left(  P\right)
=1$. Since $\alpha$ is an ultrafilter, eventually $\pi_{n}\left(  P\right)
=0$ along $\alpha$ or eventually $\pi_{n}\left(  P\right)  =1$ along $\alpha$.
Thus $\pi \left(  P\right)  =\left \{  \pi_{n}\left(  P\right)  \right \}
_{\alpha}$ is either $0$ or $1.$ Hence $P$ is either $0$ or $1$.
\end{proof}

\noindent {\bf Remark.} As was pointed out in \cite{Hadwin-Li}, the tracial ultraproducts remain unchanged when you replace the 2-norm by a p-norm ($1\le p <\infty$).
Therefore all results in this paper remain valid if 2-norms are replaced by p-norms.

\end{document}